\documentclass[12pt]{amsart}
\usepackage{fancybox}
\usepackage{ascmac}
\usepackage{delarray}
\usepackage{enumerate}
\usepackage{amsmath,amssymb}
\usepackage{color}
\usepackage{comment}

\usepackage{mathrsfs}
\usepackage{paralist}

\makeatletter
 \@addtoreset{equation}{subsection}
\makeatother

\def\re{\mathbb{R}}

\def\N{\mathbb{N}}

\def\pd{\partial}
\def\ol{\overline}

\def\la{\lambda}

\def\disp{\displaystyle}
\def\({\left(}
\def\){\right)}
\def\pd{\partial}

\def\intO{\int_{\Omega}}

\def\intRN{\int_{\re^N}}

\def\w{\omega}

\numberwithin{equation}{section}
\newtheorem{theorem}{Theorem}[section]
\newtheorem{corollary}[theorem]{Corollary}
\newtheorem{lemma}[theorem]{Lemma}
\newtheorem{proposition}[theorem]{Proposition}

\newtheorem{remark}[theorem]{Remark}

\numberwithin{theorem}{section}

\begin{document}

\title[Finsler Hardy inequalities]{Finsler Hardy inequalities}
\author[A. Mercaldo]{A. Mercaldo$^1$}
\author[M. Sano]{M. Sano$^2$}
\author[F. Takahashi]{F. Takahashi$^3$}

\date{\today}
\setcounter{footnote}{1}
\footnotetext{
Universit\`a di Napoli Federico II, Dipartimento di Matematica e Applicazioni 
`R. Caccioppoli'',
Complesso Monte S. Angelo, via Cintia, 80126 Napoli, Italy;\\
e-mail: {\tt mercaldo@unina.it}}

\setcounter{footnote}{2}
\footnotetext{
Department of Mathematics, Tokyo Institute of Technology,
2-12-1 Ookayama, Meguro-ku, Tokyo 152-8551, Japan. \\
e-mail:{\tt megumisano0609@gmail.com}}

\setcounter{footnote}{3}
\footnotetext{
Department of Mathematics, Osaka City University,
3-3-138 Sugimoto, Sumiyoshi-ku, Osaka 558-8585, Japan. \\
e-mail:{\tt futoshi@sci.osaka-cu.ac.jp}}

\begin{abstract} 
In this paper we present a unified simple approach to anisotropic Hardy inequalities in various settings. 
We consider Hardy inequalities which involve a Finsler distance from a point or from the boundary of a domain. 
The sharpness and the non-attainability of the constants in the inequalities are also proved.

\rm 
\\[0.1cm]
\rm 
\end{abstract}
\maketitle

\bigskip

\section{Introduction}
\label{section:Introduction}
The interest in the so-called anisotropic problems arose from G. Wulff's work on crystal shapes and minimization of anisotropic surface tensions in 1901 
and it is becoming increasingly important in different contexts, as in the field of phase changes and phase of separation in multiphase materials (cf. \cite{BCS(book)}, \cite{Bellettini-Paolini}). 
This justifies the necessity to extend to anisotropic case many of the classical tools, which are useful in classical variational problems. 
In this paper we are interested in sharp anisotropic Hardy-type inequalities. 
The basic idea is to endow the space $\re^N$ with the distance obtained by a Finsler metric and to extend several Hardy-type inequalities in such a new geometrical context.

The classical Hardy inequality asserts that for any $p \ge 1$, $p \ne N$, if $\Omega$ is a domain of $\re^N (N \ge 2)$ containing the origin, 
then
\begin{equation}
\label{classical Hardy}
	\left | \frac{N-p}{p} \right|^p \int_{\Omega} \frac{|u|^p}{|x|^p} dx \le \int_{\Omega} \left| \nabla u \right|^p dx
\end{equation}
holds for $ u \in C_0^{\infty}(\Omega)$ if $1 \le p < N$ and for $u \in C_0^{\infty}(\Omega \setminus \{ 0 \})$ if $p > N$.
Here the constant $\left| \frac{N-p}{p} \right|^p$ is sharp and never attained when $p > 1$. 
The critical Hardy inequality corresponding to the case $p=N$ has also been studied (cf. \cite{BFT1}, \cite{BFT2}, \cite{Ioku-Ishiwata}, \cite{Sano-TF(EJDE)}, \cite{TF});
in this case, for example, if $\Omega$ is a ball having center at the origin and radius $R$, 
then $|x|^N$ appearing in \eqref{classical Hardy} is replaced by the Hardy potential of the type $|x|^N(\log \frac{R}{|x|})^N$. 

Several variants of the Hardy inequalities \eqref{classical Hardy} have been known.
Among these we recall the {\it geometric type Hardy inequality} which asserts that, if $1 < p < \infty$ and $\Omega$ is a convex, possibly unbounded domain in $\re^N$, 
then 
\begin{equation}
\label{geo_Hardy}
	\(\frac{p-1}{p}\)^p \intO \frac{|u|^p}{(d(x))^p} dx \le \intO |\nabla u|^p dx\,, \qquad u \in C_0^{\infty}(\Omega)
\end{equation}
where $d(x)={\rm dist}(x, \partial \Omega)$ denotes the usual distance function from the boundary of $\Omega$ and the constant $\left( \frac{p-1}{p} \right)^p$ is sharp. 
An improved version of \eqref{geo_Hardy} has been proved in \cite{BFT1} 
where the best constant is given for a larger class of domains which verify the geometric assumption that $d$ is $p$-superharmonic in $\Omega$, i.e.,
\begin{equation}
\label{delta d}
	-\Delta_p d \ge 0
\end{equation}
in the distribution sense.
Here $\Delta_p$ is the $p$-Laplace operator $\Delta_p u = {\rm div}(|\nabla u|^{p-2} \nabla u)$. 

Anisotropic Hardy inequalities are also known. 
For example \eqref{classical Hardy} and \eqref{geo_Hardy} have been extended to the case where the Euclidean norm is replaced by a Finsler norm in \cite{Van Schaftingen}, \cite{Della Pietra-Blasio-Gavitone} 
when $p=2$, and \cite{Bal}, \cite{BGM}, \cite{Brasco-Franzina} when $p \ne 2$, respectively.
The method in \cite{Brasco-Franzina} and \cite{Bal} is to use Picone type identities in Finsler setting.

In this paper further anisotropic sharp Hardy inequalities will be proved.
We consider a Finsler norm $H$ and its polar function $H^0$, whose definitions are given in \S \ref{section:Notation}. 
Our first main result gives the following {\sl sharp anisotropic subcritical Hardy inequality} which is a consequence of Theorem \ref{Thm:H_p} in \S \ref{section:Hardy} 
and Theorem \ref{Thm:H_p attainability} in \S \ref{section:attainability}.

\begin{theorem}{\rm (Sharp anisotropic subcritical Hardy inequality)}
Assume  $1 \le p < N$ or  $ p> N$. Let $\Omega$ be a domain in $\re^N$. Then the following inequality 
\begin{equation}
\label{FH_p_intro}
	\left| \frac{N-p}{p} \right|^p \intO \frac{|u|^p}{\( H^0(x) \)^p} dx \le \intO \left| \frac{x}{H^0(x)} \cdot \nabla u \right|^p dx
\end{equation}
holds true for any $u \in C_0^{\infty}(\Omega)$ if $1 \le p < N$, 
and for any $u \in C_0^{\infty}(\Omega \setminus \{ 0 \}$ if $p > N$. 
Moreover if $0 \in \Omega$, 
the constant $\( \frac{N-p}{p} \)^p$ is sharp and not attained if $1 < p < N$ and the constant $N-1$ is attained for any nonnegative $H^0$-radially decreasing, compactly supported function when $p=1$.
\end{theorem}
For the notion of $H^0$-radially decreasing function, see \S \ref{section:attainability}.

The critical case $p=N$ is also studied and the following result is a consequence of Theorem \ref{Thm:H_N} and Theorem \ref{Thm:H_N attainability}.
\begin{theorem}{\rm (Sharp anisotropic critical Hardy inequality)}
Let $\Omega \subset \re^N$, $N \ge 2$, be a bounded domain containing the origin and put $R = \sup_{x \in \Omega} H^0(x)$.
Then the inequality
\begin{equation}\label{FH_N_intro}
	\( \dfrac{N-1}{N} \)^N \intO \frac{|u|^N}{(H^0(x))^N (\log \frac{R}{H^0(x)})^N} dx
	\le \intO \left| \frac{x}{H^0(x)} \cdot \nabla u \right|^N dx
\end{equation}
holds for any $u \in C_0^{\infty}(\Omega)$.  
Moreover the constant $\( \frac{N-1}{N} \)^N$ is sharp and not attained. 
\end{theorem}

In \S \ref{section:geo Hardy} an {\sl anisotropic Hardy inequality of geometric type} is proved, 
while the attainability of the best constant is also studied in \S \ref{section:attainability}.
\begin{theorem}{\rm (Anisotropic Hardy inequality of geometric type)}
Let $1 < p < \infty$ and suppose $\delta = \delta(x)$ is a nonnegative, $\Delta_{H,p}$-superharmonic function on $\Omega$, i.e.,
\begin{equation}
\label{delta_p d}
	-\Delta_{H,p} \delta \ge 0
\end{equation}
in the distributional sense, where
\[
	\Delta_{H,p} \delta(x) = {\rm div} \( H^{p-1}(\nabla \delta(x))(\nabla H)(\nabla \delta(x)) \)
\]
denotes the Finsler $p$-Laplacian of $\delta$.
Then the inequality
\begin{equation}
\label{eq:geometric Finsler Hardy}
	\( \frac{p-1}{p} \)^p \intO \frac{|u|^p}{\delta^p} H^p(\nabla \delta) dx \le \intO |\nabla u \cdot (\nabla H)(\nabla \delta)|^p dx
\end{equation}
holds true for any $u \in C_0^{\infty}(\Omega)$.
\end{theorem}

Condition \eqref{delta_p d} will be discussed in \S \ref{section:geo Hardy}. 
Here we just remark that \eqref{delta_p d} coincides with \eqref{delta d} when we choose the Euclidean norm as the Finsler norm 
and $d$ as $\delta$ in \eqref{delta_p d}.

In Section \S \ref{section:Weighted} a weighted Finsler-Hardy-Poincar\'e inequality have been proved with respect to a weight $\rho$ which satisfies suitable assumptions. 

Our Hardy inequalities will be proved by using a simple unified approach valid for any choice of Hardy potential. 
A related approach has been adopted in \cite{BFT1}, \cite{BFT2}.

Finally we fix our attention on two anisotropic Hardy inequalities (\ref{FH_p_intro}) and (\ref{FH_N_intro}) which are quite different from each other in view of their forms, scaling structures and optimal constants. 
However, according to \cite{Sano-Takahashi}, we can reveal an unexpected relation between the critical and the subcritical anisotropic Hardy inequalities 
and show that the critical anisotropic Hardy inequality on a ball is embedded into a family of the subcritical anisotropic Hardy inequalities on the whole space 
by using a transformation which connects both inequalities.
In \S \ref{section:Relation} we show that the transformation conserves not only the best constants but also the scale invariance structures of both inequalities,
at least in the $H^0$-radial setting.

\vspace{1em}\noindent
{\it Note added to Proof.}

After completing this work, the authors of this paper are informed by Professor M. Ruzhansky of his recent seminal works on the Hardy, Rellich, and other functional inequalities on homogeneous groups
with arbitrary quasi-norms \cite{ORS(arXiv2016)}, \cite{RS(arXiv2016)}, \cite{RS(Adv.Math)}, \cite{RS(CCM)}, \cite{RSY(TAMS)}, \cite{RSY(IEOT)}.
In \cite{RS(Adv.Math)}, for example, the following $L^p$-Hardy inequality 
\[
	\left\| \frac{f}{|x|} \right\|_{L^p(\mathbb{G})} \le \frac{p}{Q-p} \| \mathcal{R} f \|_{L^p(\mathbb{G})}, \quad 1 < p < Q
\]
is proved on a homogeneous group $\mathbb{G}$ with the homogeneous dimension $Q$ and a homogeneous quasi-norm $| \cdot |$.
Here the operator $\mathcal{R} = \mathcal{R}_{| \cdot |} = \frac{d}{d|x|}$ is called a radial operator. 
Other problems such as the optimality of constants and the existence of remainder terms are also studied in the above and subsequent papers.
If $\mathbb{G}$ is chosen as an abelian group $(\re^N, +)$ and $| \cdot |$ as $H^0(\cdot)$, our Hardy inequality \eqref{FH_p_intro} is nothing but the above inequality
since $\mathcal{R} f = \frac{x}{H^0(x)} \cdot \nabla f$ in this situation.
Their proof is based on the polar coordinate decomposition
\[
	\int_{\mathbb{G}} f(x) dx = \int_0^{\infty} \int_{\mathcal{S}} f(r\omega) r^{Q-1} d\sigma(\omega) dr
\] 
for $f \in L^1(\mathbb{G})$ where $\mathcal{S} = \{ x \in \mathbb{G} \, : \, |x| = 1 \}$, and is different from ours in this paper, which depends basically on the use of divergence theorem.
In this sense, many results in the present paper can be seen as special cases of the results above with different proofs.  
We stick to the Finsler setting since we want to apply our inequalities to the nonlinear problems involving the Finsler Laplacian.
Also we believe that our method of proof will be useful in such possible applications.


\section{Notation and basic properties}
\label{section:Notation}

Let $H: \re^N \to \re$ be a nonnegative, convex function of class $C^2(\re^N \setminus \{ 0 \})$,
which is even and positively homogeneous of degree 1:
\begin{equation}
\label{pos_H}
	H(t \xi ) = |t| H(\xi), \quad \forall \xi \in \re^N, \, \forall t \in \re.
\end{equation}
The above assumptions give the existence of positive constants $\alpha$ and $\beta$ such that
\[
	\alpha |\xi| \le H(\xi) \le \beta |\xi|, \quad \xi \in \re^N.
\]
Let $K$ denote the convex closed set
\[
	K = \{ \xi \in \re^N \,: \, H(\xi) \le 1 \}.
\]
Sometimes we will say that $H$ is  {\it the gauge} of $K$.
The {\it polar function} of $H$  is the function $H^0: \re^N \to \re$ 
 defined by 
\[
	H^0(x) = \sup_{\xi \in \re^N \setminus \{ 0 \}} \frac{\xi \cdot x}{H(\xi)} = \sup_{\xi \in K} (\xi \cdot x)\,, \qquad x \in \re^N\,.
\]
Throughout this paper $\xi \cdot x = \sum_{j=1}^N \xi_j x_j$ denotes the usual inner product of $\re^N$.

\noindent Note that, by definition of $H^0$, the {\sl Schwarz inequality} holds true, i.e., 
\begin{equation}
\label{Schwarz}
	|\xi \cdot x| \le H(\xi) H^0(x), \quad \forall \xi, x \in \re^N.
\end{equation}
It is well-known that $H^0$ is a convex, positively homogeneous of degree 1, continuous function on $\re^N$, 
and the following inequality is satisfied
\[
	\frac{1}{\beta} |x| \le H^0(x) \le\frac{1}{\alpha} |x|, \quad \forall x \in \re^N.
\]
Also the following equality 
\[
	H(\xi) = (H^0)^0(\xi) = \sup_{x \in \re^N \setminus \{ 0 \}} \frac{x \cdot \xi}{H^0(x)}\,, \qquad \xi \in \re^N,
\]
holds and $H^0$ itself is the gauge of the closed convex set
\[
	K^0 = \{ x \in \re^N \,: \, H^0(x) \le 1 \}.
\]
We say that $K$ and $K_0$ are polar to each other. 
The interior set of $K^0$, i.e.,
\[
	\mathcal{W} = \{ x \in \re^N \,: \, H^0(x) < 1 \}
\]
is called the {\it Wulff ball}, or $H^0$-unit ball, and we denote $\kappa_N = \mathcal{H}^N(\mathcal{W})$. 
In this case, the {\it anisotropic $H$-perimeter} of $\mathcal{W}$, denoted by $P_H(\mathcal{W}; \re^N)$, is
$P_H(\mathcal{W}, \re^N) = N \kappa_N$.
For more explanation about the anisotropic perimeter, see \cite{AFTL} and {\cite{Belloni-Ferone-Kawohl}.
Throughout the paper, we denote
\[
	\w_{N-1} = P_H(\mathcal{W}; \re^N) = N \kappa_N.
\]
Denote 
$$
	\mathcal{W}_R = \{ x \in \re^N \,|\, H^0(x) <R \}
$$
for any $R>0$ and we identify $\mathcal{W}_\infty$ with $\re^N$.
A function  $H \in C^2 \(\re^N\setminus \{0\} \)$ is a {\it Finsler  norm} if it satisfies properties \eqref{pos_H}, 
and moreover $H$ is strongly convex in the sense that the Hesse matrix of $H^2$, ${\rm Hess}(H^2)$ is positive definite. 
For references about Finsler norms (or, more in general, for Finsler metrics) see \cite{BCS(book)}, \cite{Bellettini-Paolini}.

Here we just recall further properties, whose proofs are contained in \cite{Bellettini-Paolini} Lemma 2.1, 2.2, or \cite{Van Schaftingen} Proposition 6.2.
\begin{proposition}
\label{Prop:identities} Let $H$ be a Finsler norm on $\re^N$. 
Then the following properties hold true:
\begin{enumerate}
\item[\rm(1)] $\nabla_{\xi} H(\xi) \cdot \xi = H(\xi)$, \qquad $\xi \ne 0$.
\item[\rm(2)] $\(\nabla_{\xi} H \)(t \xi) = \frac{t}{|t|} \( \nabla_{\xi} H \)(\xi)$, \qquad $\xi \ne 0, t \ne 0$.
\item[\rm(3)] $\(\nabla^2_{\xi} H \)(t \xi) = \frac{1}{|t|} \( \nabla_{\xi} H \)(\xi)$, \qquad $\xi \ne 0, t \ne 0$.
\item[\rm(4)] $H\( \nabla H^0(x) \) = 1$.
\item[\rm(5)] $H^0(x) \( \nabla_{\xi} H \) \( \nabla_x H^0(x) \) = x$.
\end{enumerate}
Similarly, following properties also hold true:
\begin{enumerate}
\item[\rm(1')] $\nabla_{x} H^0(x) \cdot x = H^0(x)$, \qquad $x \ne 0$.
\item[\rm(2')] $\(\nabla_{x} H^0 \)(t x) = \frac{t}{|t|} \( \nabla_{x} H^0 \)(x)$, \qquad $x \ne 0, t \ne 0$.
\item[\rm(3')] $\(\nabla^2_{x} H^0 \)(t x) = \frac{1}{|t|} \( \nabla_{x} H^0 \)(x)$, \qquad $x \ne 0, t \ne 0$.
\item[\rm(4')] $H^0 \( \nabla_{\xi} H(\xi) \) = 1$.
\item[\rm(5')] $H(\xi) \( \nabla_{x} H^0 \) \( \nabla_{\xi} H(\xi) \) = \xi$.
\end{enumerate}
\end{proposition}

Finally, given a smooth function $u$ on $\re^N$, {\it the Finsler Laplace operator} of $u$ (associated with $H$) is defined by
\begin{align*}
	\Delta_H u(x) &= {\rm div} \( H(\nabla u(x)) \( \nabla_{\xi} H \)(\nabla u(x)) \) \\
	&= \sum_{j=1}^N \frac{\pd}{\pd x_j} \( H(\xi) H_{\xi_j}(\xi) \Big|_{\xi = \nabla u(x)} \)
\end{align*}
and, more generally, for any $1<p<\infty$, {\it the Finsler $p$-Laplace operator} $\Delta_{H,p}$ by
\[
	\Delta_{H,p} u(x) = \mbox{div} \( H^{p-1}(\nabla u(x))(\nabla_{\xi} H)(\nabla u(x)) \).
\]
\vspace{1em}
Note that though the Finsler gradient vector 
\[
	\nabla_H u(x) = H(\nabla u(x)) \( \nabla_{\xi} H \)(\nabla u(x)) = \nabla_{\xi} \( \frac{1}{2} H^2(\xi) \) \Big|_{\xi = \nabla u(x)} 
\]
is a nonlinear operator, thanks to the strict convexity oh $H$, $\Delta_H$ and $\Delta_{H,p}$ is a uniformly elliptic operator locally.
The Finsler Laplacian has been widely investigated in literature and its notion goes back to the work of G. Wulff, who considered it to describe the theory of crystals. 
Many other authors developed the related theory in several settings, considering both analytic and geometric points of view, 
see (\cite{AIS}, \cite{Bellettini-Paolini}, \cite{Belloni-Ferone-Kawohl}, \cite{BCS}, \cite{Cianci-Salani}, \cite{Della Pietra-Blasio}, \cite{Ferone-Kawohl} and references therein).

\section{Hardy type inequalities}
\label{section:Hardy}

In this section, we prove several Finsler Hardy type inequalities in a unified method.
This simple approach is motivated by \cite{BFT1}, \cite{BFT2}, and \cite{TF}. 

\begin{theorem}{\rm (Sharp anisotropic subcritical Hardy inequality)}
\label{Thm:H_p}
Assume  $1 \le p < N$ or $ p> N$. 
Let $\Omega$ be a domain in $\re^N$. Then the following inequality holds true 
\begin{equation}
\label{FH_p}
	\left| \frac{N-p}{p} \right|^p \intO \frac{|u|^p}{\( H^0(x) \)^p} dx \le \intO \left| \frac{x}{H^0(x)} \cdot \nabla u \right|^p dx
\end{equation}
for any $u \in C_0^{\infty}(\Omega)$ if $1 \le p < N$, 
and for any $u \in C_0^{\infty}(\Omega \setminus \{ 0 \})$ if $p > N$.
\end{theorem}

\begin{remark}\rm 
The anisotropic subcritical Hardy inequality (\ref{FH_p}) is invariant under the scaling $u_\la (x) = \la^{\frac{N-p}{p}} u(\la x)$, $(\la >0)$ when $\Omega =\re^N$. 
Indeed, by (\ref{pos_H}) we can easily check the following equalities
\begin{align*}
\intRN \frac{|u_\la (x) |^p}{\( H^0(x) \)^p} dx &= \intRN \frac{|u (y) |^p}{\( H^0(y) \)^p} dy, \\
\intRN \left| \frac{x}{H^0(x)} \cdot \nabla u_\la (x) \right|^p dx &= \intRN \left| \frac{y}{H^0(y)} \cdot \nabla u (y) \right|^p dy.
\end{align*}
\end{remark}

\begin{proof} 
We just prove the assertion when $1 \le p < N$, since the proof of the case where $p > N$ is similar.
Define
\[
	F(x) = \frac{x}{\( H^0(x) \)^{\la}}, \quad x \in \Omega.
\]
Then we have
\begin{align*}
	{\rm div} F(x) &= \frac{N}{\( H^0(x) \)^{\la}} + (-\la) \(H^0(x) \)^{-\la-1} \nabla_x H^0(x) \cdot x \\
	&= \frac{N-\la}{\( H^0(x) \)^{\la}},
\end{align*}
since $\nabla H^0(x) \cdot x = H^0(x)$ by Proposition \ref{Prop:identities} (1').
Take $\la = p$.
Then for any $u \in C_0^{\infty}(\Omega)$, 
we compute
\begin{align*}
	&\intO \frac{|u(x)|^p}{\( H^0(x) \)^p} dx 
	= \left| \frac{1}{N-p} \intO |u(x)|^p \mbox{div} F(x) dx \right| \\
	&= \left| - \frac{1}{N-p} \intO \nabla (|u|^p ) \cdot \frac{x}{\( H^0(x) \)^p} dx \right| \\
	&= \left| - \frac{p}{N-p} \intO |u|^{p-2} u \nabla u  \cdot \frac{x}{\( H^0(x) \)^p} dx \right| \\
	&\le \left| \frac{p}{N-p} \right| \( \intO \frac{|u|^p}{\( H^0(x) \)^p} dx \)^{(p-1)/p} \( \intO \left| \frac{x}{\( H^0(x) \)} \cdot \nabla u \right|^p dx \)^{1/p}.
\end{align*}
This yields the conclusion.
\end{proof}

\begin{remark}\rm 
Note that by \eqref{Schwarz} and the positively 1-homogeniety of $H^0$, 
the right-hand side of the inequality \eqref{FH_p} is estimated as
\[
	\intO \left| \frac{x}{H^0(x)} \cdot \nabla u \right|^p dx  \le \intO \( H( \nabla u(x)) \)^p dx.
\]
Thus Theorem \ref{Thm:H_p} improves the following inequality by Van Schaftingen (\cite{Van Schaftingen} Proposition 7.5), 
\[
	\left| \frac{N-p}{p} \right|^p \intO \frac{|u|^p}{\( H^0(x) \)^p} dx\le \intO \( H( \nabla u(x)) \)^p dx,
\]
which is obtained by the use of symmetrization.
\end{remark}
\bigskip
Next result concerns the critical case $p=N$.

\begin{theorem}{\rm (Sharp anisotropic critical Hardy inequality)}
\label{Thm:H_N}
Let $\Omega \subset \re^N$, $N \ge 2$, be a bounded domain and put $R = \sup_{x \in \Omega} H^0(x)$.
Then the inequality
\begin{equation}\label{FH_N}
	\( \dfrac{N-1}{N} \)^N \intO \frac{|u|^N}{(H^0(x))^N (\log \frac{R}{H^0(x)})^N} dx
	\le \intO \left| \frac{x}{H^0(x)} \cdot \nabla u \right|^N dx
\end{equation}
holds for any $u \in C_0^{\infty}(\Omega)$. 
\end{theorem}

\begin{remark}\rm
The anisotropic critical Hardy inequality (\ref{FH_N}) is invariant under the scaling $u_\la (x) = \la^{-\frac{N-1}{N}} u\( \( \frac{H^0(x)}{R}\)^{\la -1} x \)$, $(\la >0)$ 
when $\Omega =\mathcal{W}_R = \{ x \in \re^N \,: \, H^0(x) < R \}$. 
For the proof, see \S \ref{section:scale FH_N}.
\end{remark}

\begin{proof}
Define
\[
	G(x) = \frac{x}{\( H^0(x) \)^N \( \log \frac{R}{H^0(x)} \)^{\la}}, \quad x \in \Omega.
\]
Note that
\[
	{\rm div} \( \frac{x}{(H^0(x))^N} \) = 0 
\]
by the former calculation in the proof of Theorem \ref{Thm:H_p}.
Then we have
\begin{align*}
	\mbox{div}\, G(x) &=  \frac{x}{(H^0(x))^N} \cdot (-\la) \( \log \frac{R}{H^0(x)} \)^{-\la-1} \( -\frac{1}{H^0(x)} \) \nabla H^0(x) \\ 
	&= \la \frac{x \cdot \nabla H^0(x)}{(H^0(x))^{N+1}} \( \log \frac{R}{H^0(x)} \)^{-\la-1} \\
	&= \frac{\la}{(H^0(x))^N \( \log \frac{R}{H^0(x)} \)^{\la +1}}.
\end{align*}
In particular, by choosing $\la = N-1$, we have
\begin{align*}
	&\int_{\Omega} \frac{|u(x)|^N}{(H^0(x))^N \( \log \frac{R}{H^0(x)} \)^N} dx \\ 
	&= \left| \frac{1}{N-1} \int_{\Omega} |u|^N \mbox{div} F(x) dx \right| \\
	&= \left| - \frac{1}{N-1} \int_{\Omega} \nabla (|u|^N ) \cdot \frac{x}{(H^0(x))^N \( \log \frac{R}{H^0(x)} \)^{N-1}} dx \right| \\
	&= \left| - \frac{N}{N-1} \int_{\Omega} |u|^{N-2} u \nabla u  \cdot \frac{\frac{x}{H^0(x)}}{(H^0(x))^{N-1} \( \log \frac{R}{H^0(x)} \)^{N-1}} dx \right| \\
	&\le  \frac{N}{N-1} \( \int_{\Omega} \frac{|u|^N}{(H^0(x))^N \( \log \frac{R}{H^0(x)} \)^N} dx \)^{\frac{N-1}{N}}
	\( \int_{\Omega} \left| \frac{x}{H^0(x)} \cdot \nabla u \right|^N dx \)^{\frac{1}{N}}
\end{align*}
for $u \in C_0^{\infty}(\Omega)$. This yields the conclusion.
\end{proof}


\section{Hardy inequality of geometric type}
\label{section:geo Hardy}

Let $\Omega$ be a domain in $\re^N$ with Lipschitz boundary and let 
\begin{equation}
\label{d_H}
	d_H(x) = \inf_{y \in \pd\Omega} H^0(x-y)
\end{equation}
be {\it the anisotropic distance} of $x \in \ol{\Omega}$ to the boundary of $\Omega \subset \re^N$.
Then we have
\[
	H(\nabla d_H(x)) = 1 \quad a.e. \ \text{in} \ \Omega
\]
and 
\[
	\frac{1}{\beta} d(x) \le d_H(x) \le \frac{1}{\alpha} d(x)
\]
where $d(x) = \inf_{y \in \pd\Omega} |x-y|$ is the Euclidean distance from the boundary $\pd\Omega$.
In \cite{Della Pietra-Blasio-Gavitone}, the authors studied the {\it anisotropic Hardy inequality of geometric type} as follows:
\begin{theorem}{\rm (\cite{Della Pietra-Blasio-Gavitone})}
\label{Thm:geometric H_2}
Suppose $d_H$ is a $\Delta_H$-superharmonic in $\Omega$, i.e.,
\[
	-\Delta_H d_H \ge 0
\]
in the distribution sense.
Then the inequality
\begin{equation}
\label{eq:DPBG}
	\frac{1}{4} \intO \frac{|u|^2}{(d_H(x))^2} dx \le \intO (H(\nabla u))^2 dx
\end{equation}
holds true for any $u \in C_0^{\infty}(\Omega)$.
\end{theorem}
Note that if $\Omega$ is convex, the assumption $-\Delta_H d_H \ge 0$ holds true. 
In \cite{Della Pietra-Blasio-Gavitone}, it is shown that there exists a non-convex domain $\Omega$ such that $d_H$ is $\Delta_H$-superharmonic on $\Omega$.
For the Euclidean geometric type Hardy inequalities, see \cite{Brezis-Marcus}, \cite{FTT(JEMS)}, \cite{MMP}, \cite{Tidblom(PAMS)}, and references there in.

In the next theorem, we improve their result in the following form:
\begin{theorem}{\rm (Anisotropic $L^p$-Hardy inequality of geometric type)}
\label{Thm:geometric H_p}
Let $1 < p < \infty$ and suppose $\delta = \delta(x)$ is a nonnegative, $\Delta_{H,p}$-superharmonic function on $\Omega$, i.e.,
\[
	-\Delta_{H,p} \delta \ge 0
\]
in the distributional sense. 
Then the inequality
\begin{equation}
\label{eq:geometric Finsler Hardy}
	\( \frac{p-1}{p} \)^p \intO \frac{|u|^p}{\delta^p} H^p(\nabla \delta) dx \le \intO |\nabla u \cdot (\nabla_{\xi} H)(\nabla \delta)|^p dx
\end{equation}
holds true for any $u \in C_0^{\infty}(\Omega)$.
\end{theorem}

\begin{remark}\label{rem:geom}\rm
Note that by (\ref{Schwarz}) and Proposition \ref{Prop:identities} (4'), we have
\[
	|\nabla u \cdot (\nabla_{\xi} H)(\nabla \delta)| \le H(\nabla u) H^0((\nabla_{\xi} H)(\nabla \delta)) = H(\nabla u).
\]
Thus, by \eqref{eq:geometric Finsler Hardy}, we have the following inequality
\begin{equation}
\label{eq:geometric Finsler Hardy_d}
	\( \frac{p-1}{p} \)^p \intO \frac{|u|^p}{(d_H(x))^p} dx \le \intO (H(\nabla u))^p dx.
\end{equation}
Moreover if the domain satisfies the assumption $-\Delta_H d_H \ge 0$ in the distribution sense
(this is the case if $\Omega$ is convex), 
then taking $\delta = d_H$ and using $H(\nabla d_H) = 1$ a.e. in $\Omega$, 
we have $-\Delta_{H,p} d_H = -\Delta_{H} d_H \ge 0$.
Thus  Theorem \ref{Thm:geometric H_p} is an improvement of the result proved in \cite{Della Pietra-Blasio-Gavitone}, 
which gives the inequality \eqref{eq:geometric Finsler Hardy_d} when $p = 2$ under the assumption $-\Delta_{H} d_H \ge 0$ .
\end{remark}
\bigskip

\begin{proof}
For $x \in \Omega$, define
\[
	F(x) = \frac{(H(\nabla \delta(x)))^{p-2}}{\delta(x)^{p-1}} \nabla \delta(x).
\] 
Then we have
\begin{align*}
	H(F)(\nabla_{\xi} H)(F) =  \frac{(H(\nabla \delta))^{p-1}}{\delta^{p-1}} (\nabla_{\xi} H)(\nabla \delta)
\end{align*}
and
\begin{align*}
	&\mbox{div} (H(F)(\nabla_{\xi} H)(F)) \\
	&= \delta^{1-p} \mbox{div} \( (H(\nabla \delta))^{p-1} (\nabla_{\xi} H)(\nabla \delta) \) 
	+ (1-p) \delta^{-p} \nabla \delta \cdot (H(\nabla \delta))^{p-1} (\nabla_{\xi} H)(\nabla \delta) \\ 
	&= \frac{\Delta_{H,p} \delta}{\delta^{p-1}} - (p-1) \frac{(H(\nabla \delta))^p}{\delta^p}.
\end{align*}
Then for $u \in C_0^{\infty}(\Omega)$, we have
\begin{align*}
	&\int_{\Omega} |u|^p \( \frac{\Delta_{H,p} \delta}{\delta^{p-1}} - (p-1) \frac{(H(\nabla \delta))^p}{\delta^p} \) dx \\ 
	&= \int_{\Omega} |u|^p \mbox{div} (H(F)(\nabla_{\xi} H)(F)) dx \\ 
	&= -p \int_{\Omega} |u|^{p-1} (sgn(u)) \nabla u \cdot  (H(F) (\nabla_{\xi} H)(F) dx \\ 
	&= -p \int_{\Omega} |u|^{p-1} (sgn(u)) \nabla u \cdot \frac{(H(\nabla \delta))^{p-1}}{\delta^{p-1}} (\nabla_{\xi} H)(\nabla \delta) dx.
\end{align*}
Now, by the assumption $\Delta_{H,p} \delta \le 0$, we see
\begin{align*}
	0 &\ge \int_{\Omega} |u|^p \( 0 - (p-1) \frac{(H(\nabla \delta))^p}{\delta^p} \) dx \\ 
	&\ge -p \int_{\Omega} |u|^{p-1} (sgn(u)) \nabla u \cdot  \frac{(H(\nabla \delta))^{p-1}}{\delta^{p-1}} (\nabla_{\xi} H)(\nabla \delta) dx,
\end{align*}
and, by the H\"older inequality, this leads to
\begin{align*}
	(p-1) \int_{\Omega} |u|^p \frac{(H(\nabla \delta))^p}{\delta^p} dx &\le p \int_{\Omega} |u|^{p-1} \Big| \nabla u \cdot (\nabla_{\xi} H)(\nabla \delta) \Big| \frac{(H(\nabla \delta))^{p-1}}{\delta^{p-1}} dx\\
	&\le p \left (\int_{\Omega} |u|^{p}\frac{(H(\nabla \delta))^p}{\delta^p} dx\right)^\frac{p-1}{p}
	\left (\int_{\Omega}\Big| \nabla u \cdot (\nabla_{\xi} H)(\nabla \delta) \Big|^p	\right)^\frac{1}{p}.
\end{align*}
This gives us the result.
\end{proof}

Now some consequences of Theorem \ref{Thm:geometric H_p} are proved. 
The first one is an anisotropic Hardy inequality when $\Omega$ is the half-space
\[
	\re^N_{+} = \{ x = (x_1,x_2, \cdots, x_N) \, : \, x_N > 0 \}
\]
with $N \ge 2$.
Denote $d_H(x) = \inf_{y \in \pd \re^N_{+}} H^0(x-y)$ the anisotropic distance from the boundary for $x \in \re^N_{+}$. 
Note that $H(\nabla d_H(x)) = 1$ a.e. and,  by the convexity of $\re^N_{+}$,  $-\Delta_{H,p} d_H \ge 0$ holds for $1 < p < \infty$ and $N \ge 2$.
By Theorem \ref{Thm:geometric H_p} and the Remark \ref{rem:geom}, we have {\it the anisotropic Hardy inequality on the half-space}:

\begin{corollary}{\rm (Anisotropic Hardy inequality on the half-space)}
\label{cor:geometric H_p}
Assume $1 < p < \infty$ and $N \ge 2$. 
Then the inequality
\begin{equation}
\label{H_p half}
	\(\frac{p-1}{p}\)^p \int_{\re^N_{+}} \frac{|u|^p}{\( d_H(x) \)^p} dx \le \int_{\re^N_{+}} |\nabla u \cdot (\nabla_{\xi} H)(\nabla d_H)|^p dx
\end{equation}
for any $u \in C_0^{\infty}(\re^N_{+})$. 
\end{corollary}

The second  consequence of Theorem \ref{Thm:geometric H_p} is a lower bound for $\la_{1,p}(\Omega)$, the first eigenvalue of the Finsler $p$-Laplacian $\Delta_{H,p}$, by means of anisotropic inradius. 
Define 
\[
	\la_{1,p}(\Omega) = \min_{u \in W^{1,p}_0(\Omega), u \ne 0} \frac{\intO (H(\nabla u))^p dx}{\intO |u|^p dx}
\]
and assume that $\tau_H$, {\it the anisotropic inradius} of $\Omega$, is finite, i.e.,
\[
	\tau_H = \sup_{x \in \Omega} d_H(x) < \infty.
\]
We prove the following result:
\begin{corollary}
\label{Cor:lower bound}
Let $\Omega$ be a bounded domain in $\re^N$ satisfying $-\Delta_{H} d_H \ge 0$ in the distribution sense and $\tau_H < \infty$.
Let $\la_{1,p}(\Omega)$ be the first eigenvalue of the Finsler $p$-Laplacian $\Delta_{H,p}$.
Then it holds that
\[
	\la_{1,p}(\Omega) \ge \(\frac{p-1}{p}\)^p \(\frac{1}{\tau_H}\)^p. 
\]
\end{corollary}

\begin{proof}
Applying \eqref{eq:geometric Finsler Hardy_d} to the first eigenfunction $\phi$ of $\la_{1,p}(\Omega)$, normalized as $\| \phi \|_{L^p(\Omega)} = 1$, we obtain
\[
	\la_{1,p}(\Omega) = \intO (H(\nabla \phi))^p dx \ge \( \frac{p-1}{p} \)^p \intO \frac{|\phi|^p}{(d_H(x))^p} dx \ge \( \frac{p-1}{p} \)^p \frac{1}{(\tau_H)^p}.
\]
\end{proof}


\section{Weighted Finsler-Hardy-Poincar\'e inequalities}
\label{section:Weighted}

In this section, we prove weighted version of Finsler Hardy-Poincar\'e type inequalities on $\re^N$, following arguments in \cite{Kombe-Ozaydin(2009)} and \cite{Kombe-Ozaydin(2013)}. 

\begin{theorem}{\rm (Weighted anisotropic Hardy-Poincar\'e inequality)}
\label{Thm:WFH_p}
Let $1 \le p < \infty$. 
Assume there exists a nonnegative function $\rho$ on $\re^N \setminus \{0\}$ such that $H(\nabla \rho) = 1$ and $\Delta_H \rho \ge \frac{C}{\rho}$ in the sense of distributions
where $C > 0$.
Then for $\alpha \in \re$ such that $C + \alpha > -1$, the following inequality holds true 
\begin{equation}
\label{WFH_p}
	\( \frac{C + \alpha + 1}{p} \)^p \int_{\re^N} \rho^{\alpha} |u|^p dx \le \int_{\re^N} \rho^{\alpha+p} \left| (\nabla_{\xi} H)(\nabla \rho) \cdot \nabla u \right|^p dx
\end{equation}
for any $u \in C_0^{\infty}(\re^N \setminus \{ 0 \}$.
\end{theorem}

\begin{proof}
From the assumptions, we see $\nabla \rho \cdot (\nabla_{\xi} H)(\nabla \rho) = H(\nabla \rho) = 1$ and 
\[
	\rho \Delta_H \rho = \rho \, {\rm div} ((\nabla_{\xi} H)(\nabla \rho)) \ge C.
\]
Thus
\begin{align*}
	{\rm div}(\rho (\nabla_{\xi} H(\nabla \rho)) = \rho \, {\rm div}((\nabla_{\xi} H)(\nabla \rho)) + \nabla \rho \cdot (\nabla_{\xi} H)(\nabla \rho) \ge C + 1.
\end{align*}
Multiplying this inequality by $\rho^{\alpha} |u|^p$ and integrating over $\re^N$, we have
\[
	(C + 1) \intRN \rho^{\alpha} |u|^p dx \le \intRN {\rm div}(\rho (\nabla_{\xi} H(\nabla \rho)) \rho^{\alpha} |u|^p dx.
\]
The divergence theorem and the H\"older inequality implies that
\begin{align*}
	(C + \alpha + 1) \intRN \rho^{\alpha} |u|^p dx &\le \left| -p \intRN \rho^{\alpha+1} |u|^{p-2} u \nabla u \cdot (\nabla_{\xi} H(\nabla \rho)) dx \right| \\
	&\le p \( \intRN \rho^{\alpha} |u|^p dx \)^{\frac{p-1}{p}} \( \intRN \rho^{\alpha+p} |\nabla u \cdot (\nabla_{\xi} H(\nabla \rho))|^p dx \)^{\frac{1}{p}}.
\end{align*}
After some manipulations, we have \eqref{WFH_p}.
\end{proof}

\begin{remark} \rm 
\label{remark:H^0}
Since by Proposition \ref{Prop:identities} (4),  $\rho(x) = H^0(x)$ satisfies that $H(\nabla \rho) = 1$ and $\Delta_H \rho = \frac{N-1}{\rho}$, we have from Theorem \ref{Thm:WFH_p} that
\begin{align*}
	\( \frac{N + \alpha}{p} \)^p \int_{\re^N} (H^0(x))^{\alpha} |u|^p dx \le \int_{\re^N} (H^0(x))^{\alpha+p} \left| \frac{x}{H^0(x)} \cdot \nabla u \right|^p dx
\end{align*}
for any $u \in C_0^{\infty}(\re^N \setminus \{ 0 \})$.
This improves Theorem 5.4 by Brasco-Franzina \cite{Brasco-Franzina}, because the right-hand side of the above inequality is less than or equal to
$\int_{\re^N} (H^0(x))^{\alpha+p} H^p(\nabla u) dx$.

\noindent Finally we recall that in the Euclidean case, i.e. $\rho(x)=H_0(x)=|x|$,  weighted  Hardy-Poincar\'e inequalities are well-known (see, for example,  \cite{ACP}, \cite{ABCMP1}, \cite{Sano-TF(DIE)} and references therein)
\end{remark}

{\it The Uncertainty Principle} in quantum mechanics, sometimes called Heisenberg-Pauli-Weyl inequality, is well known in Euclidean context and it claims that
\[
	\frac{N^2}{4} \( \intRN |f(x)|^2 dx \)^2 \le \( \intRN |x|^2 |f(x)|^2 dx \) \( \intRN |\nabla f(x)|^2 dx \) 
\]
for any $f \in C_0^{\infty}(\re^N)$.
In Finsler context, we obtain the following:

\begin{theorem}{\rm (Anisotropic uncertainty principle inequality)}
\label{Thm:FUP}
Let $1 \le p < \infty$ and $N \ge 2$. 
Assume there exists a nonnegative function $\rho$ on $\re^N \setminus \{0\}$ such that $H(\nabla \rho) = 1$ and $\Delta_H \rho \ge \frac{C}{\rho}$ in the sense of distributions
where $C > 0$.
Then the following inequality holds true: 
\begin{equation}
\label{FUP}
	\( \frac{C + 1}{2} \)^2 \( \int_{\re^N} |u|^2 dx \)^2 \le \( \int_{\re^N} \rho^2 |u|^2 dx \) \(\int_{\re^N} \left| (\nabla_{\xi} H)(\nabla \rho) \cdot \nabla u \right|^2 dx \)
\end{equation}
for any $u \in C_0^{\infty}(\re^N)$.
Especially, we have
\[
	\( \frac{N}{2} \)^2 \( \int_{\re^N} |u|^2 dx \)^2 \le \( \int_{\re^N} (H^0(x))^2 |u|^2 dx \) \(\int_{\re^N} \left| \frac{x}{\nabla (H^0(x))} \cdot \nabla u \right|^2 dx \)
\]
for any $u \in C_0^{\infty}(\re^N)$.
\end{theorem}

\begin{proof}
Since $H(\nabla \rho^2) = 2\rho H(\nabla \rho)$ and $(\nabla_{\xi} H)(\nabla \rho^2) = (\nabla_{\xi} H)(\nabla \rho)$ by Proposition \ref{Prop:identities},
by using assumptions, we have
\[
	\Delta_H (\rho^2) = {\rm div} (2 \rho H(\nabla \rho) (\nabla_{\xi} H)(\nabla \rho)) = 2H^2(\nabla \rho) + 2\rho \Delta_H \rho \ge 2 + 2C.
\]
Multiplying this inequality by $|u|^2$ and integrating over $\re^N$, we have
\[
	\intRN \Delta_H (\rho^2) |u|^2 dx \ge 2(1 + C) \intRN |u|^2 dx.
\]
On the other hand, integration by parts and Schwarz inequality implies
\begin{align*}
	\intRN \Delta_H (\rho^2) |u|^2 dx &= \intRN {\rm div} (2 \rho (\nabla_{\xi} H)(\nabla \rho)) |u|^2 dx \\
	&= - \intRN 2 \rho (\nabla_{\xi} H)(\nabla \rho) \cdot 2 u \nabla u dx \\
	&\le 4 \( \intRN \rho^2 |u|^2 dx \)^{\frac{1}{2}} \(\intRN |(\nabla_{\xi} H)(\nabla \rho) \cdot \nabla u |^2 dx \)^{\frac{1}{2}}.
\end{align*}
After some computations, we have \eqref{FUP}.
The last claim follows from Remark \ref{remark:H^0} and \eqref{FUP}.
\end{proof}


\section{The best constant and its attainability on Hardy inequalities} 
\label{section:attainability}

In this section we investigate the sharpness and the attainability of the constants in anisotropic Hardy inequalities in the previous sections. 
We call a function $f$ defined on $\re^N$ is {\it $H^0$-radial} if there exists a function $F = F(r)$ defined on $\re_{+}$ such that $f(x) = F(H^0(x))$.  
If $F$ is decreasing on $\re_{+}$, then $f$ is called $H^0$-radially decreasing.
For a function space $X$, we define $X_{H^0rad} = \{ u \in X \,:\, \text{$u$ is $H^0$-radial} \}$.
Admitting some ambiguity, we sometimes write $f(x) = f(r)$ with $r = H^0(x)$ for $H^0$-radial function $f$.
Let us begin with two preliminary results.

\begin{proposition}
\label{Prop:spherical}
Let $R \in (0, +\infty]$.
For any $u \in W_0^{1,p}(\mathcal{W}_R)$, 
there exists a $H^0$-radial function $U \in W_0^{1,p}(\mathcal{W}_R)$ such that the followings hold true: 
\begin{align}
\label{simm_1}
	\int_{\mathcal{W}_R} V(H^0(x)) |U|^p \,dx &= \int_{\mathcal{W}_R} V(H^0(x)) |u |^p dx, \\
\label{simm_2}
	\int_{\mathcal{W}_R}  |\nabla U|^p \,dx &\le \int_{\mathcal{W}_R} \left| \nabla u \cdot \frac{x}{H^0(x)} \right|^p dx
\end{align}
where $V = V(r)$ is any function on $[0, R]$. 
\end{proposition}

\begin{proof}
For $x \in \re^N \setminus \{0\}$, let us write $x=r\w$ where $r=H^0(x)$, $\w \in \pd \mathcal{W}$ and $\w_{N-1}= P_H(\mathcal{W}; \re^N) = N \kappa_N$.
It is enough to show Proposition \ref{Prop:spherical} for $u \in C_0^1(\mathcal{W}_R)$ by the density argument.
For any $u \in C_0^1(\mathcal{W}_R)$, we set
\begin{align*}
	\ol{U}(r) = \(\w_{N-1}^{-1} \int_{\pd \mathcal{W}} |u(r \w)|^p dS_{\w} \)^{\frac{1}{p}},
\end{align*}
where $dS_{\w}$ denotes a measure on $\pd \mathcal{W}$ such that
$\int_{\pd \mathcal{W}} dS_{\w} = P_H(\mathcal{W}; \re^N) = \w_{N-1}$ holds true.
Then by the H\"older inequality, we have
\begin{align*}
	\left| \ol{U}'(r) \right| &= \w_{N-1}^{-\frac{1}{p}} \( \int_{\pd \mathcal{W}} |u (r \w)|^p dS_{\w} \)^{\frac{1}{p}-1} \left| \int_{\pd \mathcal{W}} |u(r \w)|^{p-2} u(r\w) \frac{\pd u}{\pd r}(r\w) dS_{\w} \right| \\
	&\le \(\w_{N-1}^{-1} \int_{\pd \mathcal{W}} \left| \frac{\pd u}{\pd r}(r \w) \right|^p dS_{\w} \)^{\frac{1}{p}}.
\end{align*}
Therefore for $U(x) = \ol{U}(H^0(x))$, we obtain
\begin{align*}
	\int_{\mathcal{W}_R} |\nabla U|^p dx &= \w_{N-1} \int_0^R |\ol{U}'(r)|^p r^{N-1} dr \\
	&\le \int_0^R \int_{\pd \mathcal{W}} \left| \frac{\pd u}{\pd r}(r \w) \right|^p r^{N-1} dr dS_{\w} \\
	&= \int_{\mathcal{W}_R} \left |\nabla u \cdot \frac{x}{H^0(x)} \right|^p dx < \infty.
\end{align*}
Thus $U \in W^{1,p}_0(\mathcal{W}_R)$ and \eqref{simm_2} is proved.
Moreover we obtain
\begin{align*}
\int_{\mathcal{W}_R} V(H^0(x)) |U|^p \,dx &= \w_{N-1} \int_0^R V(r) |U(r)|^p r^{N-1} \,dr \\
&= \int_0^R \int_{\pd \mathcal{W}} V(r) |u (r\w )|^p r^{N-1} dr dS_{\w} \\
&= \int_{\mathcal{W}_R} V(H^0(x)) |u |^p dx.
\end{align*}
Hence \eqref{simm_1} is proved.
\end{proof}

\begin{proposition}
\label{Prop:radial lemma}
For $R \in (0, +\infty]$, let $U \in C^1(0, R)$ with $U(R) := \lim_{r \to R} U(r) = 0$. 
Then the following pointwise estimates hold for any $r \in (0,R)$.
\begin{align}
\label{subcritical}
	|U(r)| &\le \( \frac{N-p}{p-1} \)^{\frac{p-1}{p}} \( \int_r^R |U'(s)|^p s^{N-1}\,ds \)^{\frac{1}{p}} r^{-\frac{N-p}{p}}, \quad (1< p < N), \\
\label{critical}
	|U(r)| &\le \( \int_r^R |U'(s)|^N s^{N-1}\,ds \)^{\frac{1}{N}} \( \log \frac{R}{r} \)^{\frac{N-1}{N}}. 
\end{align}
\end{proposition}

\begin{proof}
Since
\begin{align*}
	|U(r)| &= \left| - \int_r^R U'(s) \,ds \right| \\
	&\le \( \int_r^R |U'(s)|^p s^{N-1}\,ds \)^{\frac{1}{p}} \( \int_r^R s^{-\frac{N-1}{p-1}}\,ds \)^{\frac{p-1}{p}},
\end{align*}
we obtain (\ref{subcritical}) and (\ref{critical}).
\end{proof}

For $\Omega \subseteq \re^N$ and $1 \le p < N$,
let us define the best constant of the anisotropic subcritical $L^p$-Hardy inequality on $\Omega$ as 
\begin{equation}
\label{H_p}
	H_p(\Omega) = \inf_{0 \not\equiv u \in W_0^{1,p}(\Omega )} \frac{\disp \intO \left| \frac{x}{H^0(x)} \cdot \nabla u \right|^p dx}{\disp \intO \frac{|u|^p}{H^0(x)^p} dx}.
\end{equation}
In order to prove Theorem \ref{Thm:H_p attainability} below, we need the following result.

\begin{lemma}
\label{lem:regularity}
If $u(x) = U(H^0(x)) \in W_0^{1,p}(\re^N)$ is a $H^0$-radial minimizer of \eqref{H_p} with $\Omega = \re^N$, then $U \in C^1(0, \infty)$.
\end{lemma}

\begin{proof}
The proof is similar to that of Lemma 2.4 in \cite{Ioku-Ishiwata}.
Take any $0< a < b < \infty$.  Note that $U \in L^p(a,b)$. 
By the characterization of one dimensional Sobolev space $W^{1,p}(a,b)$, $r \mapsto U(r)$ is locally absolutely continuous. 
Particularly, $U(r)$ is differentiable for almost all $r \in (a,b)$. 
Put $f(r)=r^{N-1}| \pd_r U(r)|^{p-2} \pd_r U(r)$, $g(r)= H_p(\re^N) r^{-p+N-1} |U(r)|^{p-2} U(r)$. 
From the weak form of the Euler Lagrange equation of $H_p(\re^N)$:
\begin{align*}
	\int_0^\infty f(r) \pd_r v \,dr = \int_0^\infty g(r) v \,dr \quad ({}^{\forall}v \in W_{0, H^0rad}^{1,p}(\re^N )),
\end{align*}
we see that $f$ is weakly differentiable and its weak derivative is $\pd_r f(r) = -g(r)$ a.e. $r \in (a,b)$.
Moreover we have
\begin{align*}
	\int_a^b |\pd_r f(r)| \,dr &= H_p(\re^N) \int_a^b r^{N-p-1} |U(r)|^{p-1}\,dr < \infty.
\end{align*}
Therefore $f \in W^{1,1}(a,b)$ which yields that $f$ is absolutely continuous on $(a,b)$. 
Since $a,b >0$ are arbitrary, we see that $U \in C^1(0, \infty)$.
\end{proof}

The first main result of this section is Theorem \ref{Thm:H_p attainability} below which concerns the sharpness of the constant in the anisotropic subcritical Hardy inequality given in Theorem \ref{Thm:H_p}
\begin{theorem}
\label{Thm:H_p attainability}
Let $\Omega$ be a domain in $\re^N$ with $0 \in \Omega$ and $1 \le p < N$. Then $H_p(\Omega)$ in \eqref{H_p} is $H_p(\Omega) = \( \frac{N-p}{p} \)^p$.
Moreover $H_p(\Omega)$ is not attained if $1 < p < N$. 
On the other hand, $H_1(\Omega) = N-1$ is attained by any nonnegative function which is $H^0$-radially decreasing on its support.
\end{theorem}

\begin{proof}
Let $\delta >0$ satisfy $\mathcal{W}_{2\delta} \subset \Omega$ and $\alpha < \frac{N-p}{p}$. 
Set
\begin{align*}
	\varphi_\alpha (x) = 
	\begin{cases}
	(H^0(x))^{-\alpha} \quad &\text{if} \,\,H^0(x) \le \delta, \\
	\delta^{-\alpha -1} (2\delta -H^0(x))  &\text{if} \,\, \delta < H^0(x) < 2 \delta, \\
	0 &\text{if} \,\,2 \delta \le H^0(x).
\end{cases}
\end{align*}
For $x$ with $H^0(x) \le \delta$, by Proposition \ref{Prop:identities}. (1'), we have 
\begin{align*}
	\frac{x}{H^0(x)} \cdot \nabla \varphi_\alpha (x) 
	= -\alpha (H^0(x))^{-\alpha -1} \frac{x}{H^0(x)} \cdot \nabla H^0(x)
	= -\alpha (H^0(x))^{-\alpha -1}
\end{align*}
and
\begin{align*}
	\intO \left| \frac{x}{H^0(x)} \cdot \nabla \varphi_\alpha \right|^p dx
	&=\alpha^p \int_{\mathcal{W}_\delta} (H^0(x))^{-\alpha p -p} dx + C(\delta ) \\
	&=\alpha^p \w_{N-1} \int_0^\delta r^{-\alpha p -p +N -1} dr + C(\delta ) \\
	&=\alpha^p \w_{N-1} (N-p-\alpha p)^{-1} \delta^{N-\alpha p -p} + C(\delta).
\end{align*}
This yields that
\begin{align}
\label{nume}
	&\intO \left| \frac{x}{H^0(x)} \cdot \nabla \varphi_\alpha \right|^p dx \notag \\
	&= \alpha^p \w_{N-1} p^{-1} \( \frac{N-p}{p} - \alpha \)^{-1} \delta^{N-\alpha p -p} + o\( \( \frac{N-p}{p} - \alpha \)^{-1} \) 
\end{align}
as $\alpha \nearrow \frac{N-p}{p}$. 
On the other hand, we have
\begin{align}
\label{deno}
	&\intO \frac{|\varphi_\alpha|^p}{(H^0(x))^p} dx = \int_{\mathcal{W}_\delta} (H^0(x))^{-\alpha p -p} dx + C(\delta ) \notag \\
	&=\w_{N-1} p^{-1} \( \frac{N-p}{p} - \alpha \)^{-1} \delta^{N-\alpha p -p} + o\( \( \frac{N-p}{p} - \alpha \)^{-1} \).
\end{align}
From (\ref{nume}), (\ref{deno}) and Theorem \ref{Thm:H_p}, we see that
\begin{align*}
\( \frac{N-p}{p} \)^p \le H_p(\Omega) \le \frac{\intO \left| \frac{x}{H^0(x)} \cdot \nabla \varphi_\delta \right|^p dx}{\intO \frac{|\varphi_\alpha |^p}{(H^0(x))^p} dx} = \alpha^p + o(1) = \( \frac{N-p}{p} \)^p + o(1).
\end{align*}
Hence $H_p(\Omega)=(\frac{N-p}{p} )^p$.

Next we shall show the attainability of $H_1(\Omega)$.
For any $u \in C_0^\infty (\Omega)$, there exists $R > 0$ such that ${\rm supp} (u) \subset \mathcal{W}_R$.
Since $u$ is nonnegative $H^0$-radially decreasing function on its support, we obtain
\begin{align*}
	&\intO \frac{|u|}{H^0(x)} dx = \int_{\mathcal{W}_R} \frac{u(x)}{H^0(x)} dx 
	=\w_{N-1} \int_0^R u(r) r^{N-2} dr \\
	&= -\frac{\w_{N-1}}{N-1} \int_0^R u'(r) r^{N-1} dr =\frac{1}{N-1} \intO \left| \frac{x}{H^0(x)} \cdot \nabla u \right| dx.
\end{align*}
Therefore we see that $H_1(\Omega) = N-1$ is attained by $u$.

Finally we show the non-attainability of $H_p(\Omega)$ when $1<p<N$.
Assume by contradiction that $\tilde{u} \in W_0^{1,p}(\Omega)$ is a minimizer of $H_p(\Omega) = (\frac{N-p}{p})^p = H_p(\re^N)$. 
Then by zero-extension there exists a minimizer $\overline{u} \in W_0^{1,p}(\re^N)$ of $H_p(\re^N)$. 
By Proposition \ref{Prop:spherical}, there also exists a $H^0$-radial minimizer $u \in W_0^{1,p}(\re^N)$ of $H_p(\re^N)$. 
Write $u(x) = U(H^0(x))$.
From Lemma \ref{lem:regularity} we see that $u \in C^1(\re^N \setminus \{ 0\})$.
Now we set
\[
	J(u) = \intRN \left| \frac{x}{H^0(x)} \cdot \nabla u \right|^p dx - \( \frac{N-p}{p} \)^p \intRN \frac{|u|^p}{(H^0(x))^p} dx
\]
and consider $v(x) = V(H^0(x))= (H^0(x))^{\frac{N-p}{p}} U(H^0(x))$.
Note that $\lim_{r \to \infty} |V(r)| =0$ from Proposition \ref{Prop:radial lemma}. 
Indeed
\begin{align*}
	\lim_{r \to \infty} |V(r)| = \lim_{r \to \infty} r^{\frac{N-p}{p}}|U(r)| \le C \lim_{r \to \infty} \| H( \nabla u)\|_{L^p(\re^N \setminus \mathcal{W}_r)} =0.
\end{align*}
Since
\[
	\nabla u(x) = -\frac{N-p}{p} (H^0(x))^{-\frac{N}{p}} \nabla H^0(x) v(x) + (H^0(x))^{-\frac{N-p}{p}} \nabla v(x),
\]
we have
\begin{equation}
\label{radial H_p}
	\left| \frac{x}{H^0(x)} \cdot \nabla u \right|^p 
	= \left| -\frac{N-p}{p} (H^0(x))^{-\frac{N}{p}} v(x) + (H^0(x))^{-\frac{N-p}{p}} \frac{x}{H^0(x)} \cdot \nabla v(x) \right|^p.
\end{equation}
By recalling the inequality $|a+b|^p \ge |a|^p + p|a|^{p-2}a b$, $(p>1, a,b \in \re)$ and that the equality holds iff $b=0$, 
we see
\begin{align*}
	&\left| \frac{x}{H^0(x)} \cdot \nabla u \right|^p \\
	&\ge \( \frac{N-p}{p} \)^p (H^0(x))^{-N} |v|^p - p \( \frac{N-p}{p} \)^{p-1} |v|^{p-2} v \nabla v \cdot \frac{x}{H^0(x)} (H^0(x))^{-N+1}.
\end{align*}
Therefore we have
\begin{align*}
	J(u) &\ge \( \frac{N-p}{p} \)^p \intRN \frac{|v|^p}{(H^0(x))^N} dx - \( \frac{N-p}{p} \)^{p-1} \intRN \nabla ( |v|^p ) \cdot \frac{x}{(H^0(x))^N} dx \\
	&\hspace{10em}-\( \frac{N-p}{p} \)^p \intRN \frac{|u|^p}{(H^0(x))^p} dx. \\
	&= -\lim_{R \to \infty}\( \frac{N-p}{p} \)^{p-1} \int_{\mathcal{W}_R} \nabla ( |v|^p ) \cdot \frac{x}{(H^0(x))^N} dx \\
	&= -\lim_{R \to \infty}\( \frac{N-p}{p} \)^{p-1} \frac{|V(R)|^p}{R^N} \int_{\pd \mathcal{W}_R} x \cdot \nu dS_x=0
\end{align*}
since $\lim_{R \to \infty} |V(R)| =0$, where $\nu$ is an outer normal vector and we have used the fact ${\rm div} \( \frac{x}{(H^0(x))^N} \) = 0$.
Since $u$ is a $H^0$-radial minimizer, $J(u)=0$, which implies
\[  
	(H^0(x))^{-\frac{N-p}{p}} \frac{x}{H^0(x)} \cdot \nabla v(x) = 0
\]
by \eqref{radial H_p}.
This yields that $v(x)$ is a constant and $u(x) = c (H^0(x))^{-\frac{N-p}{p}}$ for some $c \in \re$ for $x \in \re^N \setminus \{ 0\}$. 
However $(H^0(x))^{-\frac{N-p}{p}} \not\in W_0^{1,p}(\re^N)$. This is a contradiction. 
Hence $H_p(\Omega)$ is not attained if $1 < p < N$.
\end{proof}

\begin{theorem}
\label{Thm:H_N attainability}
Let $\Omega$ be a bounded domain in $\re^N \,(N\ge 2)$, $0 \in \Omega$ and $R=\sup_{x \in \Omega} H^0(x)$. 
Then
\begin{equation*}
	H_N(\Omega) := \inf_{0 \not\equiv u \in W_0^{1,N}(\Omega)} 
	\frac{\displaystyle\intO \left| \frac{x}{H^0(x)} \cdot \nabla u \right|^N dx}{\displaystyle\intO \frac{|u|^N}{H^0(x)^N (\log \frac{R}{H^0(x)})^N} dx} 
	= \( \frac{N-1}{N} \)^N.
\end{equation*}
Moreover $H_N(\Omega)$ is not attained.
\end{theorem}

\begin{proof}
Let $\delta >0$ satisfy $\mathcal{W}_{2\delta} \subset \Omega$ and $\alpha < \frac{N-1}{N}$. Set
\begin{align*}
\varphi_\alpha (x) = 
\begin{cases}
\( \log \frac{R}{H^0(x)} \)^{\alpha} \quad &\text{if} \,\,H^0(x) \le \delta, \\
\( \log \frac{R}{\delta} \)^{\alpha -1} (2\delta -H^0(x))  &\text{if} \,\, \delta < H^0(x) < 2 \delta, \\
0 &\text{if} \,\,2 \delta \le H^0(x).
\end{cases}
\end{align*}
Since for $x$ such that $H^0(x) \le \delta$, by Proposition \ref{Prop:identities}. (1') we have 
\begin{align*}
\frac{x}{H^0(x)} \cdot \nabla \varphi_\alpha (x) 
	&= -\alpha \( \log \frac{R}{H^0(x)} \)^{\alpha -1} \frac{1}{H^0(x)},
\end{align*}
and
\begin{align}\label{nume H_N}
	&\intO \left| \frac{x}{H^0(x)} \cdot \nabla \varphi_\alpha \right|^N dx \notag \\
	&=\alpha^N \int_{\mathcal{W}_\delta} \( \log \frac{R}{H^0(x)} \)^{\alpha N -N} \frac{1}{(H^0(x))^N} dx + C(\delta ) \notag \\
	&=\alpha^N \w_{N-1} \int_0^\delta \( \log \frac{R}{r} \)^{\alpha N -N} \frac{dr}{r} + C(\delta ) \notag \\
	&= \alpha^N \w_{N-1} N^{-1} \( \frac{N-1}{N} - \alpha \)^{-1} \( \log \frac{R}{\delta} \)^{\alpha N -N+1} + o\( \( \frac{N-1}{N} - \alpha \)^{-1} \) 
\end{align}
as $\alpha \nearrow \frac{N-1}{N}$. On the other hand, we have
\begin{align}\label{deno H_N}
	&\intO \frac{|\varphi_\alpha |^N}{(H^0(x))^N (\log \frac{R}{H^0(x)})^N} dx 
	= \int_{\mathcal{W}_\delta} \( \log \frac{R}{H^0(x)} \)^{\alpha N -N} \frac{1}{(H^0(x))^N} dx + C(\delta ) \notag \\
	&=\w_{N-1} N^{-1} \( \frac{N-1}{N} - \alpha \)^{-1} \( \log \frac{R}{\delta} \)^{\alpha N -N+1} + o\( \( \frac{N-1}{N} - \alpha \)^{-1} \).
\end{align}
From (\ref{nume H_N}), (\ref{deno H_N}) and Theorem \ref{Thm:H_N}, we see that
\begin{align*}
	\( \frac{N-1}{N} \)^N \le H_N(\Omega) 
	\le \disp{ \frac{\intO \left| \frac{x}{H^0(x)} \cdot \nabla \varphi_\alpha \right|^N dx}{\intO \frac{|\varphi_\alpha |^N}{(H^0(x))^N (\log \frac{R}{H^0(x)})^N} dx} }
	= \alpha^N + o(1) 
	= \( \frac{N-1}{N} \)^N + o(1).
	\end{align*}
Hence $H_N=(\frac{N-1}{N})^N$.

Next we shall show the non-attainability of $H_N(\Omega)$ by a contradiction. 
Assume there exists a minimizer $\tilde{u} \in W_0^{1,N}(\Omega)$ of $H_N(\Omega)$. 
Then by zero-extension there exists a minimizer $\overline{u} \in W_0^{1,N}(\mathcal{W}_R)$ of $H_N(\mathcal{W}_R)$, where $R=\sup_{x \in \Omega} H^0(x)$. 
By Proposition \ref{Prop:spherical}, there also exists a $H^0$-radial minimizer $u \in W_0^{1,N}(\mathcal{W}_R)$ of $H_N(\mathcal{W}_R)$. 
Write $u(x) = U(H^0(x))$.
We see that $u \in C^1(\mathcal{W}_R \setminus \{ 0\})$ in the same way as Lemma 2.4 in \cite{Ioku-Ishiwata}. 
Now we set
\[
	J(u) = \int_{\mathcal{W}_R} \left| \frac{x}{H^0(x)} \cdot \nabla u \right|^N dx - \( \frac{N-1}{N} \)^N \int_{\mathcal{W}_R} \frac{|u|^N}{H^0(x)^N (\log \frac{R}{H^0(x)})^N} dx
\]
and consider $v(x) = V(H^0(x)) = (\log \frac{R}{H^0(x)})^{-\frac{N-1}{N}} U(H^0(x))$. 
Note that $|V(R)| =\lim_{r \to R} |V(r)| =0$ from Proposition \ref{Prop:radial lemma}. 
Indeed
\begin{align*}
	\lim_{r \to R} |V(r)| = \lim_{r \to R} \( \log \frac{R}{r} \)^{-\frac{N-1}{N}}|U(r)| \le C \lim_{r \to R} \| H( \nabla u)\|_{L^N(\mathcal{W}_R \setminus \mathcal{W}_r)} =0.
\end{align*}
Since
\[
	\nabla u(x) = \frac{N-1}{N} \( \log \frac{R}{H^0(x)} \)^{-\frac{1}{N}} \frac{\nabla H^0(x)}{H^0(x)} v(x) + \( \log \frac{R}{H^0(x)} \)^{-\frac{N-1}{N}} \nabla v(x),
\]
we have
\begin{align}
\label{radial H_N}
	&\left| \frac{x}{H^0(x)} \cdot \nabla u \right|^N \\
	&= \left| \frac{N-1}{N} \( \log \frac{R}{H^0(x)} \)^{-\frac{1}{N}} \frac{v(x)}{H^0(x)} + \( \log \frac{R}{H^0(x)} \)^{\frac{N-1}{N}} \frac{x}{H^0(x)} \cdot \nabla v(x) \right|^N. \notag
\end{align}
By recalling the inequality $|a+b|^N \ge |a|^N + N|a|^{N-2}a b$, $(N>1, a,b \in \re)$ and that the equality holds iff $b=0$, 
we see
\begin{align*}
	\left| \frac{x}{H^0(x)} \cdot \nabla u \right|^N 
	&\ge \( \frac{N-1}{N} \)^N \frac{|v|^N}{(H^0(x))^N} \( \log \frac{R}{H^0(x)} \)^{-1} \\
	&- N \( \frac{N-1}{N} \)^{N-1} |v|^{N-2} v \nabla v \cdot \frac{x}{H^0(x)} (H^0(x))^{-N+1}.
\end{align*}
Therefore we have
\begin{align*}
	J(u) &\ge \( \frac{N-1}{N} \)^N \int_{\mathcal{W}_R} \frac{|v|^N}{(H^0(x))^N} \( \log \frac{R}{H^0(x)} \)^{-1} dx \\
	&- \( \frac{N-1}{N} \)^{N-1} \int_{\mathcal{W}_R} \nabla (|v|^N ) \cdot \frac{x}{(H^0(x))^N} dx \\
	&- \( \frac{N-1}{N} \)^N \int_{\mathcal{W}_R} \frac{|u|^N}{H^0(x)^N (\log \frac{R}{(H^0(x))})^N} dx \\
	&=-  \( \frac{N-1}{N} \)^{N-1} \frac{|V(R)|^N}{R^N} \int_{\pd \mathcal{W}_R} x \cdot \nu dS_x = 0
\end{align*}
by $|V(R)| = 0$.
Since $u$ is a $H^0$-radial minimizer, $J(u)=0$, which implies
\[  
	 \( \log \frac{R}{H^0(x)} \)^{\frac{N-1}{N}} \frac{x}{H^0(x)} \cdot \nabla v(x)=0
\]
by \eqref{radial H_N}.
This in turn yields that $v(x)$ is a constant and  $u(x)=c \( \log \frac{R}{H^0(x)} \)^{\frac{N-1}{N}}$ for some $c \in \re$ for $x \in \mathcal{W}_R \setminus \{ 0\}$. 
However $\( \log \frac{R}{H^0(x)} \)^{\frac{N-1}{N}} \not\in W_0^{1,N}(\mathcal{W}_R)$. 
This is a contradiction and $H_N(\Omega)$ is not attained.
\end{proof}

%
%

\begin{theorem}
\label{Thm:H_p-half best constant}
Let $N \ge 2$ and $1 < p < \infty$. Define $d_H$ as in \eqref{d_H}.
Then
\begin{equation*}
	C_p(\re^N_{+}) := \inf_{0 \not\equiv u \in W_0^{1,p}(\re^N_{+})} \frac{\displaystyle\int_{\re^N_{+}} |\nabla u \cdot (\nabla_{\xi} H)(\nabla d_H)|^p dx}{\displaystyle\int_{\re^N_{+}} \frac{|u|^p}{(d_H(x))^p} dx} 
	= \( \frac{p-1}{p} \)^p.
\end{equation*}
\end{theorem}

\begin{proof}
The inequality (\ref{H_p half}) implies that $C_p(\re^N_{+}) \ge \( \frac{p-1}{p} \)^p$.
For $R > 0$, let 
\begin{equation}
\label{QR}
	Q_R = \{ x = (x^{\prime}, x_N) \in \re^N_{+} \, | \, |x^{\prime}| < R,  0 < x_N < R \}
\end{equation}
be an open cube and
let $\eta$ be a smooth cut-off function with $0 \le \eta \le 1$, $\eta \equiv 1$ on $Q_R$, $\eta \equiv 0$ on $(Q_{2R})^c$. 
For $\alpha > \frac{p-1}{p}$, put $u_{\alpha}(x) = \eta(x) (d_H(x))^\alpha$.
Then $u_{\alpha} \in W^{1,p}_0(\re^N_{+} \cap Q_{2R})$ and
\begin{align*}
	\nabla u_{\alpha} = \alpha \eta (d_H)^{\alpha-1} \nabla d_H + (\nabla \eta) (d_H)^{\alpha}.
\end{align*}
Thus 
\begin{align*}
	&\nabla u_{\alpha} = \alpha (d_H)^{\alpha-1} \nabla d_H \quad \text{on} \quad \re^N_{+} \cap Q_R, \\
	&\int_{\re^N_{+} \cap Q_R} H^p(\nabla u_{\alpha}) dx = \alpha^p \int_{\re^N_{+} \cap Q_R} (d_H(x))^{p(\alpha-1)} dx, \\
	&\int_{\re^N_{+} \cap Q_R} \frac{|u_{\alpha}|^p}{(d_H(x))^p} dx = \int_{\re^N_{+} \cap Q_R} (d_H(x))^{p(\alpha-1)} dx.
\end{align*}
Now, since the inequality $\frac{1}{\alpha_2} |x| \le H^0(x) \le \frac{1}{\alpha_1} |x|$ holds, we have
$\frac{1}{\alpha_2} d_E(x) \le d_H(x) \le \frac{1}{\alpha_1} d_E(x)$, 
where $d_E(x) = x_N$ denotes the Euclidean distance of $x \in \re^N_{+}$ from the boundary: $d_E(x) = \inf_{y \in \pd \re^N_{+}} |x-y|$. 
Since
\[
	\int_{\re^N_{+} \cap Q_R} x_N^{p(\alpha-1)} dx = \int_{|x^{\prime}| < R} \int_0^R  x_N^{p(\alpha-1)} dx_N dx^{\prime} = C(R) \frac{R^{p(\alpha-1) + 1}}{p(\alpha-1) + 1},
\]
where $C(R) = \int_{|x^{\prime}| < R} dx^{\prime}$,
we have
\[
	\int_{\re^N_{+} \cap Q_R} (d_H(x))^{p(\alpha-1)} dx = O\( \frac{1}{p(\alpha-1) + 1} \) \quad \text{as} \quad \alpha \searrow \frac{p-1}{p}
\]
for any fixed $R > 0$.
On the other hand, by the convexity of $H$ and the fact $H(\nabla d_H) = 1$, we have
\begin{align*}
	H(\nabla u_{\alpha}) &= H\( \alpha \eta (d_H)^{\alpha-1} \nabla d_H + (\nabla \eta) (d_H)^{\alpha} \) \\
	&\le H(\alpha \eta (d_H)^{\alpha-1} \nabla d_H) + H((d_H)^{\alpha} (\nabla \eta)) \\
	&\le \alpha (d_H)^{\alpha-1} + (d_H)^{\alpha} H(\nabla \eta).
\end{align*}
Since $(2R)^{p(\alpha-1) +1} - R^{p(\alpha-1) +1} \to 0$ as $\alpha \searrow \frac{p-1}{p}$,
we have
\begin{align*}
	&\int_{\re^N_{+} \cap (Q_{2R} \setminus Q_R)} x_N^{p(\alpha-1)} dx = \int_{R < |x^{\prime}| < 2R} \int_{R}^{2R}  x_N^{p(\alpha-1)} dx_N dx^{\prime} \\
	&= D(R) \frac{(2R)^{p(\alpha-1) +1} - R^{p(\alpha-1) +1}}{p(\alpha-1) + 1} = o(\frac{1}{p(\alpha-1)+1})
\end{align*}
as $\alpha \searrow \frac{p-1}{p}$ for any fixed $R>0$, where $D(R) = \int_{R < |x^{\prime}| < 2R} dx^{\prime}$.
Also we have
\begin{align*}
	&\int_{\re^N_{+} \cap (Q_{2R} \setminus Q_R)} x_N^{p\alpha} dx = \int_{R < |x^{\prime}| < 2R} \int_{R}^{2R}  x_N^{p\alpha} dx_N dx^{\prime} = O(1). 
\end{align*}
Then 
\begin{align*}
	&\int_{\re^N_{+} \cap (Q_{2R} \setminus Q_R)} (d_H(x))^{p(\alpha-1)} dx =  o(\frac{1}{p(\alpha-1)+1}), \\
	&\int_{\re^N_{+} \cap (Q_{2R} \setminus Q_R)} (d_H(x))^{p\alpha} dx =  O(1) 
\end{align*}
as $\alpha \searrow \frac{p-1}{p}$ for any fixed $R>0$ and thus
\begin{align*}
	&\int_{\re^N_{+} \cap (Q_{2R} \setminus Q_R)} H(\nabla u_{\alpha})^p dx \\
	&\le 2^{p-1} \alpha^p \int_{\re^N_{+} \cap (Q_{2R} \setminus Q_R)} (d_H)^{p(\alpha-1)} dx + 2^{p-1}  \sup_{x \in Q_{2R}} H^p(\nabla \eta(x)) \int_{\re^N_{+} \cap (Q_{2R} \setminus Q_R)} (d_H)^{p\alpha} dx \\
	&= o(\frac{1}{p(\alpha-1)+1}) + O(1).
\end{align*}
Therefore, we see
\begin{align*}
	C_p(\re^N_{+} \cap Q_{2R}) &\le \frac{\disp\int_{\re^N_{+} \cap Q_{2R}} |\nabla u_{\alpha} \cdot (\nabla_{\xi} H)(\nabla d_H)|^p dx}{\disp\int_{\re^N_{+} \cap Q_{2R}} \frac{|u_{\alpha}|^p}{(d_H(x))^p} dx} \\
	&\le \frac{\displaystyle\int_{\re^N_{+} \cap Q_R} H^p(\nabla u_{\alpha}) dx 
	+ \disp \int_{\re^N_{+} \cap (Q_{2R} \setminus Q_R)} H^p(\nabla u_{\alpha}) dx}{\disp\int_{\re^N_{+} \cap Q_R} \frac{|u_{\alpha}|^p}{(d_H(x))^p} dx} \\ 
	&= \frac{\alpha^p \disp\int_{\re^N_{+} \cap Q_R} (d_H(x))^{p(\alpha-1)} dx + \disp\int_{\re^N_{+} \cap (Q_{2R} \setminus Q_R)} H^p(\nabla u_{\alpha}) dx}{\disp\int_{\re^N_{+} \cap Q_R} (d_H(x))^{p(\alpha-1)} dx} \\
	&= \alpha^p + \frac{o(\frac{1}{p(\alpha-1)+1}) + O(1)}{O(\frac{1}{p(\alpha-1)+1})}
\end{align*}
as $\alpha \searrow \frac{p-1}{p}$.
Then taking the limit $\alpha \searrow \frac{p-1}{p}$, we have $C_p(\re^N_{+}) \le (\frac{p-1}{p})^p$.
Thus we have proven the result.
\end{proof}

\begin{remark}\rm
Let $\Omega \subset \re^N_{+}$ be a domain with a flat boundary portion on $\pd\re^N_{+}$,
that is,
\[
	\text{$Q_{4R} \subset \Omega$ for some $R > 0$}
\]
where $Q_{4R}$ be an open cube as in (\ref{QR}).
Then we have $C_p(\Omega) = \( \frac{p-1}{p} \)^p$.
Because for such domain, 
$d_E(x) = \inf_{y \in \pd\Omega} |x-y| = x_N$ for $x \in Q_{2R} \cap \Omega$
and the same proof as Theorem \ref{Thm:H_p-half best constant} works well.
\end{remark}

\begin{remark}\rm
Let $\Omega$ be a domain in $\re^N$ satisfying that $d_H$ is weakly twice differentiable and $-\Delta d_H \ge 0$ a.e.in $\Omega$,
where $d_H(x) = \inf_{y \in \pd\Omega} H^0(x-y)$.
Concerning the attainability of the best constant of (\ref{eq:geometric Finsler Hardy_d}), i.e.,
\begin{equation}
\label{C_p}
	C_p(\Omega) := \inf_{0 \not\equiv u \in C_0^{\infty}(\Omega)} \frac{\displaystyle\intO |\nabla u \cdot (\nabla_{\xi} H)(\nabla d_H)|^p dx}{\displaystyle\intO \frac{|u|^p}{(d_H(x))^p} dx}, 
\end{equation}
we will have the following observation.

First, for $u \in C_0^{\infty}(\Omega)$, define $v(x) = u(x) d_H^{-(\frac{p-1}{p})}$, $v$ is a Lipschitz function and $v = 0$ on $\pd\Omega$. 
We compute
\begin{align*}
	&\nabla u = \(\frac{p-1}{p}\) d_H^{-\frac{1}{p}} v(x) \nabla d_H + d_H^{\frac{p-1}{p}} \nabla v, \\
	&\nabla u \cdot (\nabla_{\xi} H)(\nabla d_H) = \(\frac{p-1}{p}\) d_H^{-\frac{1}{p}} v(x) \nabla d_H \cdot (\nabla_{\xi} H)(\nabla d_H) \\ 
	&\hspace{7em} + d_H^{\frac{p-1}{p}} \nabla v \cdot (\nabla_{\xi} H)(\nabla d_H).
\end{align*}
Thus
\begin{align*}
	&|\nabla u \cdot (\nabla_{\xi} H)(\nabla d_H)|^p = \Big|\(\frac{p-1}{p}\) d_H^{-\frac{1}{p}} v(x) + d_H^{\frac{p-1}{p}} \nabla v \cdot (\nabla_{\xi} H)(\nabla d_H) \Big|^p \\
	&\ge \(\frac{p-1}{p}\)^p d_H^{-1} |v|^p + p \(\frac{p-1}{p}\)^{p-1} |v|^{p-2} v \nabla v \cdot (\nabla_{\xi} H)(\nabla d_H) \\
	&= \(\frac{p-1}{p}\)^p \frac{|u|^p}{d_H^p} + \(\frac{p-1}{p}\)^{p-1} \nabla (|v|^p) \cdot (\nabla_{\xi} H)(\nabla d_H), \\
\end{align*}
where, as before, we have used the fact that $|a + b|^p \ge |a|^p + p|a|^{p-2}ab$ for $p > 1$ and $a, b \in \re$.
Note that the equality holds true if and only if $b = 0$.
Thus we have
\begin{align*}
	J(u) &:= \intO |\nabla u \cdot (\nabla_{\xi} H)(\nabla d_H)|^p dx - \( \frac{p-1}{p} \)^p \intO \frac{|u|^p}{(d_H)^p} dx \\
	&= \(\frac{p-1}{p}\)^{p-1} \intO \nabla (|v|^p) \cdot (\nabla_{\xi} H)(\nabla d_H) dx \\
	&= -\(\frac{p-1}{p}\)^{p-1} \intO |v|^p (\Delta_H d_H) dx,
\end{align*}
since $H(\nabla d_H) = 1$ a.e.
Therefore, if $J(u) = 0$ for some $u \in C_0^{\infty}(\Omega)$, then we must have
\begin{align*}
	\begin{cases}
	&\Delta_H d_H = 0 \quad a.e. \ \text{in} \, \Omega, \\
	&|\nabla v \cdot (\nabla_{\xi} H)(\nabla d_H)| = 0 \quad a.e. \ \text{in} \, \Omega,
	\end{cases}
\end{align*}
since we assume that $\Delta_H d_H \le 0$ a.e.in $\Omega$.
In particular, we can claim that 
{\it if $C_p(\Omega) = \( \frac{p-1}{p} \)^p$ and $-\Delta_H d_H > 0$ on a positive measure, then $C_p(\Omega)$ is not attained}. 
\end{remark}

Let $1 < p < \infty$.
In the Euclidean case (i.e., $H(\xi) = |\xi|$, $H^0(x) = |x|$ for $\xi, x \in \re^N$), 
the following facts are known \cite{MMP} \cite{MS(NA)}: 
\begin{itemize}
\item For any convex domain $\Omega$, $C_p(\Omega) = (\frac{p-1}{p})^p$.
\item For any domain $\Omega$ such that $\pd\Omega$ has a tangent hyperplane at least one point in $\pd\Omega$, $C_p(\Omega) \le (\frac{p-1}{p})^p$.
\item For any bounded $C^2$-domain $\Omega$, if $C_p(\Omega) < (\frac{p-1}{p})^p$ then $C_p(\Omega)$ is attained.
\item For any bounded $C^2$-domain $\Omega$, $C_2(\Omega) < \frac{1}{4}$ if and only if $C_2(\Omega)$ is attained.
\end{itemize}
It could be interesting to study corresponding results for the best constant of the geometric Finsler Hardy inequality (\ref{C_p}).


\section{The scale invariance of the anisotropic critical Hardy inequality}
\label{section:scale FH_N}

In this section,
we shall show that the anisotropic critical Hardy inequality (\ref{FH_N}) is invariant under the scaling 
\[	
	u_\la (x) = \la^{-\frac{N-1}{N}} u\( \( \frac{H^0(x)}{R}\)^{\la -1} x \), \quad (\la >0)
\]
when $\Omega =\mathcal{W}_R$. 
In order to show that, we need the following lemma.
\begin{lemma}\label{Lem:scale}
Let $c > 0$ and $a \in \re$.
For $y \in \re^N$, let $x = c H^0(y)^a y$. Then the Jacobian of the transformation $y \mapsto x$ is
\[
	\left| {\rm det} \( \frac{\pd (x_1, \cdots , x_N)}{\pd (y_1, \cdots , y_N)} \) \right| = c^N (1+a) (H^0(y))^{aN}.
\]
\end{lemma}
In the special case $H(\xi)=|\xi|$, Lemma \ref{Lem:scale} is shown by \cite{Ioku-Ishiwata}.

\begin{proof}
Let us assume $y \ne 0$. Then $x \ne 0$ and we may employ ``polar coordinate" $x = r \w$, $y = \rho \w$, where $r = H^0(x)$, $\rho = H^0(y)$ and $\w \in \pd \mathcal{W}$.
By homogeneity, we see $r = H^0(x) = c (H^0(y))^{a+1} = c \rho^{a+1}$, which implies $dr = c (a+1) \rho^a d\rho$.
Also we see $dx = r^{N-1} dr dS_{\w}$, $dy = \rho^{N-1} d\rho dS_{\w}$, where $dS_{\w}$ is an $(N-1)$-dimensional measure such that
\[
	\int_{\pd \mathcal{W}} dS_{\w} = P_H(\mathcal{W}; \re^N) = \omega_{N-1} = N \kappa_N.
\]
When $\pd \mathcal{W}$ is Lipschitz,
$dS_{\w}$ can be written $dS_{\w} = H(\nu(\w)) d\mathcal{H}^{N-1}$ where $\nu(\w) = \frac{\nabla H^0(\w)}{|\nabla H^0(\w)|}$ is an unit normal vector of $\pd \mathcal{W}$. 
Now, 
\begin{align*}
	dx &= r^{N-1} dr dS_{\w} = (c \rho^{a+1})^{N-1} \frac{dr}{d\rho} d\rho dS_{\w} \\
	&= c^N (a+1) (\rho^{a+1})^{N-1} \rho^a d\rho dS_{\w} = c^N (a+1) \rho^{aN} \rho^{N-1} d\rho dS_{\w} \\
	&=  c^N (a+1) \rho^{aN} dy.
\end{align*}
On the other hand, $dx = \left| {\rm det} \( \frac{\pd (x_1, \cdots , x_N)}{\pd (y_1, \cdots , y_N)} \) \right| dy$ by definition. 
Thus we have the conclusion.
\end{proof}

\begin{remark}\rm
By a direct calculation, we see that the Jacobi matrix
\begin{align*}
	A &= \( \frac{\pd x_i}{\pd y_j} \)_{1 \le i,j \le N} = c (H^0(y))^a \left[ {\rm Id.} +  \frac{a}{H^0(y)} B \right], \\ 
	B &= ( H^0_{y_j}(y) y_i )_{1 \le i,j \le N} 
\end{align*}
has eigenvalues
\begin{itemize}
\item $c (H^0(y))^a$ with multiplicity $N-1$, whose eigenspace is the orthogonal space of the vector $\nabla H^0(y)$, $y \ne 0$.
\item $c (1 + a) (H^0(y))^a$ with multiplicity $1$, whose eigenspace is $\re y$, $y \ne 0$.
\end{itemize}
Thus actually
\[
	{\rm det} A = {\rm det} \( \frac{\pd (x_1, \cdots , x_N)}{\pd (y_1, \cdots , y_N)} \) = c^N (1+a) (H^0(y))^{aN}.
\]
\end{remark}

Set $y= \( \frac{H^0(x)}{R}\)^{\la -1} x$, that is $x=R^{1-\frac{1}{\la}} (H^0(y))^{\frac{1}{\la}-1} y$.
Since 
\begin{align*}
&\frac{\pd u(y)}{\pd x_i} \frac{x_i}{H^0(x)} = \sum_{j=1}^N \frac{\pd u(y)}{\pd y_j} \frac{\pd y_j}{\pd x_i} \frac{x_i}{H^0(x)} \\
&=\sum_{j=1}^N \frac{\pd u(y)}{\pd y_j} \frac{x_i}{H^0(x)} R^{1-\la} \left[ (\la -1) (H^0(x))^{\la -2} H^0_{x_i}(x) x_j + (H^0(x))^{\la -1} \delta_{ij} \right] \\
&=R^{1-\la} (\la -1) (H^0(x))^{\la -2} H^0_{x_i}(x) \frac{x_i}{H^0(x)} ( \nabla_y u(y) \cdot x ) + R^{1-\la} (H^0(x))^{\la -1} \frac{\pd u(y)}{\pd y_i} \frac{x_i}{H^0(x)},
\end{align*}
we obtain
\begin{align*}
&\nabla_x u(y) \cdot \frac{x}{H^0(x)} \\
&=R^{1-\la} (\la -1) (H^0(x))^{\la -1} \( \nabla_y u(y) \cdot \frac{x}{H^0(x)} \) + R^{1-\la} (H^0(x))^{\la -1} \( \nabla_y u(y) \cdot \frac{x}{H^0(x)} \) \\
&=\la R^{1-\la} (H^0(x))^{\la -1} \( \nabla_y u(y) \cdot \frac{x}{H^0(x)} \) =\la R^{\frac{1}{\la}-1} (H^0(y))^{1-\frac{1}{\la}} \( \nabla_y u(y) \cdot \frac{y}{H^0(y)} \).
\end{align*}
Therefore we see that
\begin{align*}
&\int_{\mathcal{W}_R} \left| \frac{x}{H^0(x)} \cdot \nabla u_\la (x) \right|^N dx \\
&=\la^{-N+1} \int_{\mathcal{W}_R} \la^N R^{\frac{N}{\la}-N} (H^0(y))^{N-\frac{N}{\la}} \left| \frac{y}{H^0(y)} \cdot \nabla u(y) \right|^N {\rm det} \( \frac{\pd (x_1, \cdots , x_N)}{\pd (y_1, \cdots , y_N)} \) dy \\
&=\la R^{\frac{N}{\la}-N} \int_{\mathcal{W}_R} (H^0(y))^{N-\frac{N}{\la}} \left| \frac{y}{H^0(y)} \cdot \nabla u(y) \right|^N R^{N-\frac{N}{\la}} \frac{1}{\la} (H^0(y))^{\frac{N}{\la}-N} dy \\
&=\int_{\mathcal{W}_R} \left| \frac{y}{H^0(y)} \cdot \nabla u (y) \right|^N dy,
\end{align*}
where the second equality comes from Lemma \ref{Lem:scale}, on taking $c=R^{1-\frac{1}{\la}}$ and $a=\frac{1}{\la} -1$. In the same way as above, we see that
\begin{align*}
&\int_{\mathcal{W}_R} \frac{|u_\la (x)|^N}{(H^0(x))^N (\log \frac{R}{H^0(x)})^N} dx \\
&= \la^{-N+1} \int_{\mathcal{W}_R} \frac{|u (y)|^N}{R^{N-\frac{N}{\la}} ((H^0(y)))^{\frac{N}{\la}} (\frac{1}{\la} \log \frac{R}{H^0(y)})^N} R^{N-\frac{N}{\la}} \frac{1}{\la} (H^0(y))^{\frac{N}{\la}-N} dy \\
&=\int_{\mathcal{W}_R} \frac{|u (y)|^N}{(H^0(y))^N (\log \frac{R}{H^0(y)})^N} dy.
\end{align*}
Hence the inequality (\ref{FH_N}) is invariant.

\section{Relation between the subcritical and the critical anisotropic Hardy inequalities}
\label{section:Relation}

In this section, 
according to \cite{Sano-Takahashi}, 
a relation between the critical and the subcritical anisotropic Hardy inequalities (\ref{FH_p}), (\ref{FH_N}) is presented. 
It will be shown  that the critical anisotropic Hardy inequality on a ball is embedded into a family of the subcritical anisotropic Hardy inequalities on the whole space 
by using a transformation which connects both inequalities.

\begin{theorem}
\label{Thm:H_p to H_N eq}
Let $m,N \in \N$ satisfy $N \geq 2$ and $m > N$, and $\mathcal{W}^N_R:=\{ y \in \re^N \,|\, H^0(y) <R \}$. Set
\begin{align*}
	I(u)&=\int_{\re^m} \left| \nabla u \cdot \frac{x}{H^0(x)} \right|^N \,dx - \( \frac{m-N}{N} \)^N \int_{\re^m} \frac{|u|^N}{H^0(x)^N} \,dx, \\
	J(w)&=\int_{\mathcal{W}^N_R} \left| \nabla w \cdot \frac{y}{H^0(y)} \right|^N \,dy \\
	&\hspace{5em}- \( \frac{N-1}{N} \)^N \int_{\mathcal{W}^N_R} \dfrac{|w|^N}{H^0(y)^N \( \log \frac{R}{H^0(y)} \)^N} \,dy.
\end{align*}
Then for any $w \in C^1_{H^0rad}(\mathcal{W}^N_R \setminus \{ 0\})$ (resp. $u \in C_{H^0rad}^1(\re^m \setminus \{ 0\})$),
there exists $u \in C_{H^0rad}^1(\re^m \setminus \{ 0\})$ (resp. $w \in C^1_{H^0rad}(\mathcal{W}^N_R \setminus \{ 0\})$) such that the equality
\begin{align}
\label{H_p to H_N eq}
	I(u)= \frac{\w_{m-1}}{\w_{N-1}} \( \frac{m-N}{N-1} \)^{N-1} J(u)
\end{align}
holds true where $\w_{N-1}= N \kappa_N$ and $\w_{m-1}= m \kappa_m$. 
\end{theorem}

Before the proof, we define a transformation which connects the critical and the subcritical anisotropic Hardy inequalities according to \cite{Sano-Takahashi}.
Let $m, N$ be integers such that $m > N$ and let $R > 0$ be fixed.
For a given $r \in [0,+\infty)$ (resp. $s \in [0,R)$ ), define a new variable
$s \in [0, R)$ (resp. $r \in [0, +\infty)$ ) by the relation
\begin{equation}
\label{s:r}
	\( \log \frac{R}{s} \)^{\frac{N-1}{N}} = r^{-\frac{m-N}{N}},
\end{equation}
that is,
\begin{equation}
\label{s(r):r(s)}
	s = s(r) = R \exp (-r^{-\alpha}), \quad (\text{resp.} \; r = r(s) =  \( \log \frac{R}{s} \)^{-1/\alpha} )
\end{equation}
where
\begin{equation}
\label{alpha}
	\alpha = \frac{m-N}{N-1}.
\end{equation}
Note that 
the left-hand side of (\ref{s:r}) is the virtual extremal for (\ref{FH_N}) on $\mathcal{W}^N_R$
and the right-hand side of (\ref{s:r}) is the virtual extremal for (\ref{FH_p}) on the whole space $\re^m$ when $p = N < m$.
Easy computation shows that
\begin{equation}
\label{dsdr}
	\frac{ds}{dr} = \alpha s r^{-\alpha -1} > 0,
\end{equation}
so when $r$ varies from $0$ to $+\infty$ then $s$ varies from $0$ to $R$, and vice versa.

Let $r=H^0(x)$, $x \in \re^m$ and $s=H^0(y)$, $y \in \mathcal{W}^N_R$.
Now, for a given $u = u(r) \in C_{H^0rad}^1(\re^m \setminus \{ 0\})$ (resp. $w  = w(s) \in C_{H^0rad}^1(\mathcal{W}^N_R \setminus \{ 0\})$), 
define a new function 
$w  = w(s) \in C_{H^0rad}^1(\mathcal{W}^N_R \setminus \{ 0\})$ (resp.  $u = u(r) \in C_{H^0rad}^1(\re^m \setminus \{ 0\})$) 
by 
\begin{equation}
\label{w:u}
	w(s) = u(r),
\end{equation}
where variables $s$ and $r$ are related as in (\ref{s:r}).
Note that $\lim_{s \to R} w(s) =0$ is equivalent to $\lim_{r \to \infty} u(r) =0$.
Namely, under the transformation (\ref{w:u}), the boundary $\pd \mathcal{W}^N_R$ corresponds to the infinity point $\infty$ in $\re^m$.

\begin{proof}[{\bf Proof of Theorem \ref{Thm:H_p to H_N eq}}]
Define $u$ and $w$ as in \eqref{w:u}.
Then we obtain
\begin{align*}
	\int_{\re^m} \left| \nabla u \cdot \frac{x}{H^0(x)} \right|^N dx &= \w_{m-1} \int_{0}^{\infty} |u^{\prime}(r)|^{N} r^{m-1} dr \\
	&= \w_{m-1} \int_{0}^{R} \left| w^{\prime}(s) \frac{ds}{dr} \right|^{N} r(s)^{m-1} \frac{dr}{ds} \,ds \\
	&=\w_{m-1} \int_{0}^{R} |w^{\prime}(s)|^N \( \alpha s r(s)^{-\alpha-1} \)^{N-1} r(s)^{m-1} \,ds \\ 
	&=\w_{m-1} \alpha^{N-1} \int_{0}^{R} |w^{\prime}(s)|^N s^{N-1} r(s)^{m-1 - (\alpha+1)(N-1)} \,ds \\ 
	&= \frac{\w_{m-1}}{\w_{N-1}} \alpha^{N-1} \int_{0}^{R} |w^{\prime}(s)|^N s^{N-1} \,ds \\
	&= \frac{\w_{m-1}}{\w_{N-1}} \alpha^{N-1} \int_{\mathcal{W}^N_R} \left| \nabla w \cdot \frac{y}{H^0(y)} \right|^N \,dy, 
\end{align*}
here we have used (\ref{dsdr}) and $m-1 - (\alpha +1)(N-1) = 0$ by (\ref{alpha}).

On the other hand, we have
\begin{align*}
	\int_{\re^m} \dfrac{|u(x)|^N}{H^0(x)^N} dx &= \w_{m-1} \int_{0}^{\infty} |u(r)|^N r^{m-N-1} \,dr \\
	&= \w_{m-1} \int_{0}^{R} |w(s)|^N r(s)^{m-N-1} \,\frac{dr}{ds} \,ds \\
	&= \w_{m-1} \int_{0}^{R} |w(s)|^N r(s)^{m-N-1} \alpha^{-1} s^{-1} r(s)^{\alpha+1}  \,ds \\
	&= \frac{\w_{m-1}}{\alpha} \int_{0}^{R} \frac{|w(s)|^N}{s} r(s)^{m-N+\alpha} \,ds \\
	&= \frac{\w_{m-1}}{\alpha} \int_{0}^{R} \frac{|w(s)|^N}{s (\log \frac{R}{s} )^{N}} \,ds \\
	&= \frac{\w_{m-1}}{\alpha \w_{N-1}} \int_{\mathcal{W}^N_R} \dfrac{|w(y)|^N}{H^0(y)^N \( \log \frac{R}{H^0(y)} \)^N} \,dy,
\end{align*}
since $r(s)^{m-N+\alpha} = (\log \frac{R}{s} )^{-N}$ by (\ref{s:r}) and (\ref{alpha}).

By combining these identities, we obtain Theorem \ref{Thm:H_p to H_N eq}.
\end{proof}

Lastly we show that the transformation preserves the scale invariance structures of the subcritical and the critical anisotropic Hardy inequalities.

\begin{proposition}
\label{Prop:scale}
Let $m, N$ be integers such that $m > N$.
For functions $u = u(r)$, $r \in [0, +\infty)$ and $w = w(s)$, $s \in [0,R)$,
define the scaled functions
\begin{align*}
	&u_{\mu} (r) = \mu^{\frac{m-N}{N}} u ( \mu r ), \\ 
	&w^{\la} (s) = \la^{-\frac{N-1}{N}} w (R^{1-\la} s^{\la})
\end{align*}
for $\mu, \la >0$.
Then we have 
\begin{align*}
	&w^{\la}(s(r)) = u_{\mu}(r), \quad \text{where} \; \mu = \la^{-1/\alpha}, \\
	&u_{\mu}(r(s)) = w^{\la}(s), \quad \text{where} \; \la = \mu^{-\alpha},
\end{align*}
where $s = s(r)$ and $r = r(s)$ are as in (\ref{s(r):r(s)}) and  $\alpha$ is defined in (\ref{alpha}).
\end{proposition}

\begin{proof}
By direct calculation,
\begin{align*}
	R^{1-\la} s(r)^{\la} = R^{1-\la} \( R \exp (-r^{-\alpha} ) \)^{\la} &= R \exp \( -\la r^{-\alpha} \) \\
	&= R \exp \( \( - (\la^{-1/\alpha} r) \)^{-\alpha} \) = s( \mu r ),
\end{align*}
where $\mu = \la^{-1/\alpha}$. 
Therefore we obtain
\begin{align*}
	w^{\la}(s(r)) &= \la^{-\frac{N-1}{N}} w(R^{1-\la} s(r)^{\la}) = \( \la^{-1/\alpha} \)^{\frac{m-N}{N}} w(s(\mu r)) \\ 
	&= \mu^{\frac{m-N}{N}}  u (\mu r) = u_{\mu}(r).
\end{align*}
The proof of $u_{\mu}(r(s)) = w^{\la}(s)$ for $\la = \mu^{-\alpha}$, is similar.
\end{proof}

%
%

\vspace{1em}\noindent
{\bf Acknowledgments.}

The authors of this paper thank to Prof. M. Ruzhansky for his kind comments and useful information.

A part of this work was done while the second (M.S.) and the third author (F.T.)  visited Universit\`a di Napoli Federico II, Dipartimento di Matematica e Applicazioni `R. Caccioppoli''in September, 2017.
They thank the warm hospitality of the department. 

The third author (F.T.) was supported by JSPS Grant-in-Aid for Scientific Research (B), No.15H03631.



\begin{thebibliography}{99}

\bibitem{ACP}
B.Abdellaoui, E.Colorado,  I. Peral: 
{\em Some improved Caffarelli-Kohn-Nirenberg inequalities}, 
\newblock Calc. Var. Partial Differential Equations 23 (2005), no. 3, 327-345. 

\bibitem{AIS}
G. Akagi, K. Ishige, and R. Sato:
{\em The Cauchy problem for the Finsler heat equation},
\newblock arXiv:1710.00456v1 [math.AP].

\bibitem{ABCMP1}
A. Alvino, F. Brock, F. Chiacchio, A. Mercaldo, M.R. Posteraro:
{\em Some isoperimetric inequalities on $\mathbb R^N$ with respect to weights $|x|^{\alpha}$}, 
\newblock J. Math. Anal. Appl. {\bf 451}, no. 1,  (2017), 280--318.

\bibitem{AFTL} 
A. Alvino, V. Ferone, G. Trombetti, and P. L. Lions: 
{\it Convex symmetrization and applications,}  
\newblock Ann. Inst. H. Poincar\'e Anal. Non Lin\'eaire {\bf 14} (1997), no. 2, 275--293.

\bibitem{AFT} 
A. Alvino, V. Ferone, G. Trombetti: 
{\it On the best constant in a Hardy-Sobolev inequality},
\newblock Appl. Anal. {\bf 85} (2006), no. 1-3, 171--180. 



\bibitem {BCS(book)} 
D. Bao, S.-S. Chern and Z. Shen: 
\newblock ``An introduction to Riemann-Finsler geometry", 
\newblock  Springer-Verlag, New York, 2000. xx+431 pp. 

\bibitem{Bal}
K. Bal:
{\it Hardy inequalities for Finsler $p$-Laplacian in the exterior domain,}
\newblock Mediterranean J. Math. {\bf 14}, 165 (2017), 12 pages.

\bibitem{BGM}
K. Bal, P. Garain, and I. Mandal:
{\it Some qualitative properties of Finsler $p$-Laplacian,}
\newblock Indag. Math. (N.S.)  {\bf 28} (2017),  no. 6, 1258--1264. 

\bibitem{BFT1}
G. Barbatis, S. Filippas, and A. Tertikas:
{\it A unified approach to improved $L^p$ Hardy inequalities with best constants},
\newblock Trans. Amer. Math. Soc. {\bf 356}  (2003), no. 6, 2169-2196.

\bibitem{BFT2}
G. Barbatis, S. Filippas, and A. Tertikas:
{\it  Series expansion for $L^p$  Hardy inequalities},
\newblock Indiana Univ. Math. J.  {\bf 52}  (2003),  no. 1, 171-190.


\bibitem{Bellettini-Paolini}
G. Bellettini, and M. Paolini:
{\em Anisotropic motion by mean curvature in the context of Finsler geometry,}
\newblock Hokkaido Math. J. {\bf 25} (1996), 537--566.

\bibitem{Belloni-Ferone-Kawohl}
M. Belloni, V. Ferone, and B. Kawohl:
{\em Isoperimetric inequalities, Wulff shape and related questions for strongly nonlinear elliptic operators,} 
\newblock ZAMP. {\bf 54} (2003), 771--783.

\bibitem{BCS} 
C. Bianchini, G. Ciraolo and P. Salani:
{\em An overdetermined problem for the anisotropic capacity}, 
\newblock Calc. Var. Partial Differential Equations {\bf 55} (2016), Art. 84, 24 pp.

\bibitem{Brasco-Franzina}
L. Brasco, and G. Franzina:
{\em Convexity properties of Dirichlet integrals and Picone-type inequalities,}
\newblock Kodai Math. J. {\bf 37} (2014), 769--799.

\bibitem{Brezis-Marcus}
H. Brezis, and M. Marcus:
{\em Hardy's inequalities revisited,}
\newblock Ann. Scuola Norm. Sup. Pisa Cl. Sci. (4), {\bf 25} (1997), no.1-2, 217--237.

\bibitem{Cianci-Salani}
A. Cianchi and P. Salani:
{\em  Overdetermined anisotropic elliptic problems}, 
\newblock Math. Ann. {\bf 345} (2009), 859--881.


\bibitem{Della Pietra-Blasio-Gavitone}
F. Della Pietra, G. di Blasio, and N. Gavitone:
{\em Anisotropic Hardy inequalities,}
\newblock Proc. Roy. Soc. Edinburgh Sect. A, 148A,
DOI:10.1017/S0308210517000336

\bibitem{Della Pietra-Blasio}
F. Della Pietra,  and G. di Blasio:
{\em Blow-up solutions for some nonlinear elliptic equations involving a Finsler-Laplacian}, 
\newblock Publ. Mat. {\bf 61} (2017), 213--238

\bibitem{Ferone-Kawohl}
V. Ferone, and B. Kawohl:
{\em Remarks on a Finsler-Laplacian}, 
\newblock Proc. Amer. Math. Soc. {\bf 137} (2009), no.1, 247--253.

\bibitem{FTT(JEMS)}
S. Filippas, A. Tertikas, and J. Tidblom:
{\em On the structure of Hardy-Sobolev-Maz'ya inequalities,}
\newblock  J. Eur. Math. Soc. (JEMS)  11  (2009),  no. 6, 1165--1185.


\bibitem{Ioku-Ishiwata}
N. Ioku, and N., M. Ishiwata:
{\em A scale invariant form of a critical Hardy inequality,}
\newblock Int. Math. Res. Not. (2015), no. 18, 8830--8846.

\bibitem{Kombe-Ozaydin(2009)}
I. Kombe, and M. \"Ozaydin:
{\em Improved Hardy and Rellich inequalities on Riemannian manifolds,}
\newblock Trans. Amer. Math. Soc. (2009), no. 12, 6191--6203.

\bibitem{Kombe-Ozaydin(2013)}
I. Kombe, and M. \"Ozaydin:
{\em Hardy-Poincar\'e, Rellich and uncertainty principle inequalities on Riemannian manifolds,}
\newblock Trans. Amer. Math. Soc. (2013), no. 10, 5035--5050.

\bibitem{MMP}
M. Marcus, V. J. Mizel, and Y. Pinchover:
{\em On the best constant for Hardy's inequality in $\mathbb {R}^n$,}
\newblock Trans. Amer. Math. Soc. {\bf 350} (1998), 3237--3255.

\bibitem{MS(NA)}
T. Matskewich, and P. E. Sobolevskii:
{\em  The best possible constant in generalized Hardy's inequality for convex domain in $\mathbb{R}^N$},
\newblock Nonlinear Anal. {\bf 28},  (1997),  no. 9, 1601--1610.

\bibitem{ORS(arXiv2016)}
T. Ozawa, M. Ruzhansky, and D.Suragan:
{\em $L^p$-Caffarelli-Kohn-Nirenberg type inequalities on homogeneous groups,}
\newblock arXiv:1605.02520

\bibitem{RS(arXiv2016)} 
M. Ruzhansky, and D. Suragan:
{\em Critical Hardy inequalities,}
\newblock arXiv:1602.04809

\bibitem{RS(Adv.Math)} 
M. Ruzhansky, and D. Suragan:
{\em Hardy and Rellich inequalities, identities, and sharp remainders on homogeneous groups,}
\newblock Adv. Math. {\bf 317}  (2017), 799--822. 

\bibitem{RS(CCM)}
M. Ruzhansky, and D. Suragan:
{\em Anisotropic $L^2$-weighted Hardy and $L^2$-Caffarelli-Kohn-Nirenberg inequalities,}
\newblock Commun. Contemp. Math.  {\bf 19}  (2017), no. 6, 1750014, 12 pp.

\bibitem{RSY(TAMS)}
M. Ruzhansky, D. Suragan, and N. Yessirkegenov:
{\em Extended Caffarelli-Kohn-Nirenberg inequalities, and remainders, stability, and superweights for $L^p$-weighted Hardy inequalities,}
\newblock Trans. Amer. Math. Soc. Ser. B 5  (2018), 32--62.

\bibitem{RSY(IEOT)}
M. Ruzhansky, D. Suragan, and N. Yessirkegenov:
{\em Sobolev type inequalities, Euler-Hilbert-Sobolev and Sobolev-Lorentz-Zygmund spaces on homogeneous groups,}
\newblock Integral Equations Operator Theory {\bf 90} (2018), no. 1, Art.10, 33 pp


\bibitem{Sano-TF(EJDE)}
M. Sano, and F. Takahashi: 
{\em Sublinear eigenvalue problems with singular weights related to the critical Hardy inequality,}
\newblock Electron. J. Differential Equations {\bf 2016}, Paper No. 212, 12 pp.

\bibitem{Sano-Takahashi}
M. Sano, and F. Takahashi: 
{\em Scale invariance structures of the critical and the subcritical Hardy inequalities and their improvements,}  
\newblock Calc. Var. Partial Differential Equations {\bf 56} (2017), no. 3, Paper No. 69, 14 pp.


\bibitem{Sano-TF(DIE)}
M. Sano, and F. Takahashi: 
{\em Some improvements for a class of the Caffarelli-Kohn-Nirenberg inequalities,}
\newblock Differential Integral Equations {\bf 31},  (2018),  no. 1-2, 57--74.

\bibitem{TF}
F. Takahashi:
{\em A simple proof of Hardy's inequality in a limiting case},
\newblock Arch. Math. {\bf 104} (2015), 77--82.

\bibitem{Tidblom(JFA)}
J. Tidblom:
{\em A Hardy inequality in the half-space},
\newblock J. Funct. Anal. {\bf 221} (2005),  no. 2, 482--495.

\bibitem{Tidblom(PAMS)}
J. Tidblom:
{\em A geometrical version of Hardy's inequality for $\stackrel{\circ}{W}^{1,p}(\Omega)$},
\newblock Proc. Amer. Math. Soc.  132  (2004),  no. 8, 2265--2271.

\bibitem{Van Schaftingen}
J. Van Schaftingen:
{\em Anisotropic symmetrization,}
\newblock  Ann. Inst. H. Poincar\'e Anal. Non Lin\'eaire  {\bf 23} (2006),  no. 4, 539--565. 

\end{thebibliography}
\end{document}